\documentclass[12pt]{article}

\usepackage{amsmath}
\usepackage{amsthm}
\usepackage{amssymb}
\usepackage{cite}
\usepackage{url}
\usepackage{footmisc}
\usepackage{graphicx}
\usepackage{epstopdf}
\usepackage{chngcntr}
\usepackage{enumitem}
\usepackage{hhline}
\usepackage{array}
\usepackage{hyperref}
\usepackage{color}
\usepackage{multirow}
\usepackage{tabularx}
\usepackage{tikz}

\hypersetup{colorlinks=false}

\tikzset{every picture/.style={line width=0.75pt}} 
\usetikzlibrary{patterns}

\tikzset{
    pattern size/.store in=\mcSize, 
    pattern size = 5pt,
    pattern thickness/.store in=\mcThickness, 
    pattern thickness = 0.3pt,
    pattern radius/.store in=\mcRadius, 
    pattern radius = 1pt
}
\makeatletter
\pgfutil@ifundefined{pgf@pattern@name@_7gn8i7dj0}{
    \pgfdeclarepatternformonly[\mcThickness,\mcSize]{_7gn8i7dj0}
    {\pgfqpoint{0pt}{0pt}}
    {\pgfpoint{\mcSize+\mcThickness}{\mcSize+\mcThickness}}
    {\pgfpoint{\mcSize}{\mcSize}}
    {
    \pgfsetcolor{\tikz@pattern@color}
    \pgfsetlinewidth{\mcThickness}
    \pgfpathmoveto{\pgfqpoint{0pt}{0pt}}
    \pgfpathlineto{\pgfpoint{\mcSize+\mcThickness}{\mcSize+\mcThickness}}
    \pgfusepath{stroke}
    }
}
\makeatother

\if@twoside
\else 
\fi

\numberwithin{equation}{section}
\counterwithin{figure}{section}
\counterwithin{table}{section}

\theoremstyle{plain}
\newtheorem{theorem}{Theorem}[section]
\newtheorem{lemma}[theorem]{Lemma}
\newtheorem{proposition}[theorem]{Proposition}
\newtheorem{algorithm}  [theorem]{Algorithm}

\theoremstyle{definition}
\newtheorem{definition}{Definition}[section]
\newtheorem{assumption}{Assumption}[section]
\newtheorem{example}{Example}[section]

\theoremstyle{remark}
\newtheorem*{remark}{Remark}

\newcommand{\norm}[1]{\left\|#1\right\|}
\newcommand{\abs}[1]{\left\vert#1\right\vert}
\newcommand{\spr}[1]{\left\langle\,#1\,\right\rangle}
\newcommand{\kl}[1]{\left(#1\right)}
\newcommand{\Kl}[1]{\left\{#1\right\}}

\newcommand{\R}{\mathbb{R}} 
\newcommand{\N}{\mathbb{N}}

\newcommand{\yd}{y^{\delta}}

\newcommand{\s}{\text{s}}

\newcommand{\Rc}{\mathcal{R}}
\newcommand{\Tad}{T_\alpha^\delta}
\newcommand{\Tadt}{\mathcal{T}_\alpha^\delta}

\newcommand{\lp}{{\ell_p}}

\newcommand{\Rm}{{\R^m}}
\newcommand{\Rn}{{\R^n}}

\newcommand{\vd}{v^\delta}

\newcommand{\Rl}{\mathcal{R}_{l}}
\newcommand{\Rt}{{\mathbb{R}^2}}

\newcommand{\LtRn}{{L_2(\R^n)}}
\newcommand{\prox}{\operatorname{prox}}
\newcommand{\D}{\mathcal{D}}

\newcommand{\tto}{\rightrightarrows}

\newcommand{\Rspf}{R^{\operatorname{sp}}}
\newcommand{\RTVf}{R^{\operatorname{TV}}}
\newcommand{\BVl}{\operatorname{BV}^l}
\newcommand{\BspqRn}{{B^{s}_{p,q}(\Rn)}}
\newcommand{\BsppRn}{{B^{s}_{p,p}(\Rn)}}
\newcommand{\BspRn}{{B^{s}_{p}(\Rn)}}

\newcommand{\F}{\mathcal{F}}

\renewcommand{\S}{\mathcal{S}}

\newcommand{\oR}{(-\infty,\infty]}

\newcommand{\dist}[1]{{\rm dist}(#1)}

\newcommand{\mv}{\,\mid\,}

\newcommand{\Bmv}{\,\Big\vert\,}

\newcommand{\B}{{\cal B}}

\newcommand{\I}{{\cal I}}

\newcommand{\Sp}{{\mathcal S}}
\newcommand{\Z}{{\cal Z}}
\newcommand{\Lag}{{\cal L}}

\newcommand{\longsetto}[1]{\mathop{\longrightarrow}\limits^{#1}}
\newcommand{\skalp}[1]{\langle #1\rangle}

\newcommand{\xb}{\bar x}
\newcommand{\yb}{\bar y}

\newcommand{\lb}{\bar\lambda}

\newcommand{\AT}[2]{{\textstyle{#1\atop#2}}}
\newcommand{\xba}{{\bar x^\ast}}

\newcommand{\oo}{o}
\newcommand{\OO}{{\cal O}}
\newcommand{\argmin}{\mathop{\rm arg\,min}}

\newcommand{\cl}{{\rm cl\,}}

\newcommand{\co}{{\rm conv\,}}
\newcommand{\gph}{\mathrm{gph}\,}
\newcommand{\dom}{\mathrm{dom}\,}

\newcommand{\SCD}{SCD\ }
\newcommand{\onabla}{\overline\nabla}

\newcommand{\ssstar}{semismooth$^{*}$ }
\newcommand{\ee}[2]{{#1}^{(#2)}}
\newcommand{\ZnP}{\Z_n^{P,W}}

\newcommand{\phiFB}[1]{\varphi_{\ee \lambda{#1}}^{\rm FB}(\ee x{#1})}

\newcommand{\rge}{{\rm rge\;}}

\newlength{\myAlgBox}
\title{On SCD Semismooth$^*$ Newton methods for the efficient minimization of Tikhonov functionals with non-smooth and non-convex penalties}
\author{
Helmut Gfrerer\footnote{Johannes Kepler University Linz, Institute of Numerical Mathematics, Altenbergerstra{\ss}e~69, A-4040 Linz, Austria, (gfrerer@numa.uni-linz.ac.at)} ,
Simon Hubmer\footnote{Johannes Kepler University Linz, Institute of Industrial Mathematics, Altenbergerstra{\ss}e 69, A-4040 Linz, Austria, (simon.hubmer@jku.at)} ,
Ronny Ramlau\footnote{Johannes Kepler University Linz, Institute of Industrial Mathematics, Altenbergerstra{\ss}e 69, A-4040 Linz, Austria, (ronny.ramlau@jku.at)} \footnote{Johann Radon Institute Linz, Altenbergerstra{\ss}e 69, A-4040 Linz, Austria, (ronny.ramlau@ricam.oeaw.ac.at)} 
}

\begin{document}

\maketitle

\begin{abstract}

We consider the efficient numerical minimization of Tikhonov functionals with nonlinear operators and non-smooth and non-convex penalty terms, which appear for example in variational regularization. For this, we consider a new class of SCD semismooth$^*$ Newton methods, which are based on a novel concept of graphical derivatives, and exhibit locally superlinear convergence. We present a detailed description of these methods, and provide explicit algorithms in the case of sparsity and total-variation penalty terms. The numerical performance of these methods is then illustrated on a number of tomographic imaging problems.


\end{abstract}


\section{Introduction}\label{sect_introduction}

In this paper, we consider the problem of efficiently minimizing the finite dimensional Tikhonov functional
    \begin{equation}\label{Tikhonov_finite}
        \Tad(x) := S\kl{F(x),\yd} + \alpha R(x) \,,
    \end{equation}
where $F : \D(F) \subset \Rn \to \Rm$ is a continuously differentiable operator, $R : \Rn \to \R$ is a regularization functional, $\alpha$ is the regularization parameter, $S: \Rm\times \Rm \to \R$ is a data fidelity term, and $\yd \in \Rm$ the measured data. Tikhonov (or variational regularization) functionals like \eqref{Tikhonov_finite} are a common way of regularizing ill-posed problems of the form
    \begin{equation}\label{Gx=y}
        F(x) = \yd \,,
    \end{equation}
and thus appear consistently throughout many tomographic and medical imaging applications \cite{Engl_Hanke_Neubauer_1996,Scherzer_Grasmair_Grossauer_Haltmeier_Lenzen_2008, Louis_1989, Engl_Ramlau_2015}. Often, but not always, the finite dimensional inverse problem \eqref{Gx=y} results from the discretization of an infinite dimensional operator equation 
    \begin{equation}\label{Fx=y}
        \F(u) = \vd \,, 
    \end{equation}
where $\F : \D(\F) \subset X \to Y$ is a nonlinear operator between two real or complex Banach or Hilbert spaces $X$ and $Y$. In this case, \eqref{Tikhonov_finite} is the discrete analog to the functional
    \begin{equation}\label{Tikhonov_infinite}
        \Tadt(u) := \S(\F(u),\vd) + \alpha \Rc(u) \,,
    \end{equation}
where $\S : Y \times Y \to \R$ and $\Rc : X \to \R \cup \Kl{\infty}$. For a detailed overview on theoretical aspects such as convergence and convergence rates for variational methods like \eqref{Tikhonov_finite} and \eqref{Tikhonov_infinite}, we refer to the monographs \cite{Scherzer_Grasmair_Grossauer_Haltmeier_Lenzen_2008,Schuster_Kaltenbacher_Hofmann_Kazimierski_2012,Kaipio_Somersalo_2005,Engl_Hanke_Neubauer_1996,Engl_Ramlau_2015}. As discussed there, depending on the choice of the regularizer, different emphasis can be placed on reconstructing, e.g., sparse, smooth, or piecewise constant solutions \cite{Scherzer_Grasmair_Grossauer_Haltmeier_Lenzen_2008,Schuster_Kaltenbacher_Hofmann_Kazimierski_2012}, or ones which satisfy certain statistical properties \cite{Kaipio_Somersalo_2005}. Whatever the concrete setting, in an application one is eventually faced with numerically minimizing a discrete Tikhonov functional of the form \eqref{Tikhonov_finite}.

For this, a wide range of methods from the field of optimization can be used \cite{Bonnans_Gilbert_Lemarechal_Sagastizabal_2006,Bauschke_Combettes_2017}. Their applicability again depends on whether the operator $F$ is linear or nonlinear, whether $R$ is convex or non-convex, and whether $F$ and $R$ are continuously differentiable or not. If they are differentiable, then both first and second order methods can be used, which typically perform well without many adaptations. Popular choices are for example the classic gradient descent method with a suitably chosen stepsize, or Newton's method applied to the first-order optimality condition
    \begin{equation}\label{first_order_Tad}
        \nabla \Tad(x) = 0 \,.
    \end{equation}
For the classic Newton's method to be applicable to \eqref{first_order_Tad}, $\Tad$ and thus both $F$ and $R$ have to be twice continuously differentiable. Since this regularity requirement on $R$ is often unrealistic, one may instead use so-called semismooth Newton methods, which then only require $\nabla \Tad$ to have a generalized (or Newton) derivative \cite{Hintermueller_Ito_Kunisch_2002,Chen_Nashed_Qi_2000,QiSun93,Ulbrich_2002}.

However, in many cases $R$ is not even differentiable once, but only the subgradient $\partial R$ exists. In this case, the first-order optimality condition \eqref{first_order_Tad} becomes the inclusion
    \begin{equation}\label{first_order_optimality_nonsmooth_finite}
        0 \in \partial \, \Tad(x) = \partial \kl{ S(F(x),\yd)} + \alpha \partial R(x)\,.
    \end{equation}
This set-valued inclusion problem can no longer be tackled by standard semismooth Newton methods. If in this case $R$ is at least convex, then algorithms from convex optimization \cite{Beck_2017,Bauschke_Combettes_2017} can be used to minimize \eqref{Tikhonov_finite}. Popular examples are proximal gradient or subgradient methods, the forward-backward method, implicit splitting methods such as the Douglas-Rachford algorithm, the primal-dual method of Pock and Chambolle \cite{Chambolle_Pock_2010}, and alternating methods, among others \cite{Bauschke_Combettes_2017}. However, due to the non-differentiability of $R$, these algorithms are typically rather slow without additional modifications such as Nesterov acceleration \cite{Nesterov_1983,Beck_Teboulle_2009,Attouch_Peypouquet_2016}. Furthermore, while not always strictly necessary in practical applications, the theoretical analysis of these methods usually also require the functional $S(F(\cdot),\yd)$ to be differentiable and convex, which further limits their use.

In this paper, we thus propose a new method for the minimization of \eqref{Tikhonov_finite} for penalty functionals $R$ which are both non-smooth and non-convex. For this, we adapt the so-called \emph{subspace-containing derivative (SCD) semismooth$^*$ Newton method}, which was recently introduced in \cite{GfrManOutVal22,GfrOut21} to solve general set-valued equations of the form
    \begin{equation*}
        0 \in G(x) \,,
    \end{equation*}
where $G : \Rn \tto \Rn$ is a set-valued mapping. Based on a generalized concept of derivatives following certain graphical considerations, which enables the computation of higher order derivatives of functions which are non-differentiable in the classic sense, semismooth$^*$ Newton methods can be shown to converge locally superlinearly to stationary points under minimal assumptions. Hence, we propose to adapt these methods for the minimization of the Tikhonov functional \eqref{Tikhonov_finite} with non-smooth and non-convex regularisation functionals $R$ by applying them to the set-valued first-order optimality equation \eqref{first_order_optimality_nonsmooth_finite}. Besides a detailed description of the resulting method, including a discussion of its excellent convergence properties, we also derive explicit algorithms for specific choices of $R$, including the case of sparsity and TV penalty terms. This will in particular lead to a new efficient algorithm for $\lp$ regularization for the entire range of $0 \leq p < \infty$. Finally, we apply our derived methods to a number of tomographic imaging problems, demonstrating their overall behaviour as well as their excellent performance.

The outline of this paper is as follows: In Section~\ref{sect_background_variational}, we recall some background on variational analysis, which is necessary for discussing the SCD semismooth$^*$ Newton method in Section~\ref{sect_semismooth}. In Section~\ref{sect_semismooth_Tikhonov}, we discuss the application of this method for minimizing general Tikhononv functionals, and in Section~\ref{sect_appl_sparsity_TV}, we derive explicit forms of this method for the particular cases of sparsity and TV penalty terms. Finally, in Section~\ref{sect_numerics}, we present some numerical experiments where our derived methods are applied to a number of tomographic imaging problems, before a brief conclusion in Section~\ref{sect_conclusion}.

\section{Background on variational analysis}\label{sect_background_variational}

In this section, we recall some necessary definitions from variational analysis. First of all, given a set $\Omega\subset\R^n$ and a point $\xb\in \Omega$, the \emph{tangent cone} to $\Omega$ at $\xb$ is defined as
    \begin{equation*}
        T_\Omega(\xb):=\Kl{x\in\R^n \mv \exists \, t_k \in \R\,, t_k\downarrow 0\,, x_k \to x \, \text{ with } \, \xb+t_k x_k\in \Omega \ \forall \, k } \,.
    \end{equation*}
Next, consider a function $q:\R^n\to\oR$ and a point $\xb\in \dom q$, where
    \begin{equation*}
        \dom q :=\{x\in\R^n\mv q(x)<\infty\} \,.
    \end{equation*}
Then, the \emph{regular subdifferential} of $q$ at $\xb$ is defined by
    \begin{equation*}
        \widehat\partial q(\xb):=\Kl{x^*\in\R^n \, \Big\vert\, \liminf_{x\to\xb}\frac{q(x)-q(\xb)-\skalp{x^*,x-\xb}}{\norm{x-\xb}}\geq 0} \,, 
    \end{equation*}
while the \emph{(limiting) subdifferential} is defined by
    \begin{equation*}
        \partial q(\xb):=\Kl{x^*\in\R^n\,\big\vert\, \exists \, (x_k, x_k^*)\to(\xb,x^*) \mbox{ with }x_k^*\in\widehat \partial q(x_k)\ \forall \, k \, \mbox{ and }\lim_{k\to\infty} q(x_k)=q(\xb)} \,.
    \end{equation*}
Note that if $q$ is convex, then both the regular and the limiting subdifferential coincide with the classical definition of subdifferentials used in convex analysis \cite{Bauschke_Combettes_2017}, i.e.,
    \begin{equation*}
        \partial q(\xb) = \Kl{x^* \in \R^n \, \vert \, q(x) - q(\xb) - \spr{x^*,x-\xb} \geq 0 \,,
        \quad \forall x \in \R^n } \,.
    \end{equation*}

Next, for $\xba\in\partial q(\xb)$, the function $q$ is said to be \emph{prox-regular} at $\xb$ for $\xba$, if there exist $\epsilon> 0$ and $\rho\geq 0$ such that for all $x',x\in \B_{\epsilon}(\xb)$ with $\vert q(x)-q(\xb)\vert\leq \epsilon$ one has
    \begin{equation}\label{EqProxRegDef}
        q(x')\geq q(x)+\skalp{x^*,x'-x}-\frac \rho2 \norm{x'-x}^2\,, \quad\mbox{whenever}\quad x^*\in \partial q(x)\cap \B_\epsilon(\xba) \,.
    \end{equation}
If this holds for all $\xba \in \partial q(\xb)$, then $q$ is said to be prox-regular at $\xb$. Furthermore,  $q$ is called \emph{subdifferentially continuous} at $\xb$ for $\xba$, if for any sequence $(x_k,x_k^*)$ there holds
    \begin{equation*}
        (x_k,x_k^*)\longsetto{{\gph \partial q}}(\xb,\xba) \qquad \Longrightarrow \qquad \lim_{k\to\infty}q(x_k)=q(\xb) \,. 
    \end{equation*}
If this property holds for all $\xba \in \partial q(\xb)$, then $q$ is said to be subdifferentially continuous  at $\xb$. Note that lower semicontinuous (lsc) convex functions $q$ are both prox-regular and subdifferentially continuous at every $x$ for every $x^*\in\partial q(x)$.

Furthermore, the function $q$ is said to be \emph{prox-bounded}, if there is a $\lambda>0$ such that 
    \begin{equation*}
        \exists \, x \in\R^n \,: \quad \inf_z \Kl{ \frac1{2\lambda}\norm{z-x}^2+q(z) } > -\infty \,.
    \end{equation*}
The supremum of the set of all such $\lambda$ is denoted by $\lambda_q$, the threshold for the prox-boundedness of $q$. Finally, the \emph{proximal mapping} $ \prox_{\lambda q}:\R^n\tto\R^n$ is defined by
    \begin{equation*}
        \prox_{\lambda q}(x):=\argmin_z \Kl{ \frac1{2}\norm{z-x}^2+\lambda q(z)} \,.
    \end{equation*}

\section{The SCD semismooth* Newton method}\label{sect_semismooth}

The subspace containing derivative \emph{(SCD) semismooth* Newton method} was recently introduced in \cite{GfrOut22a} for solving set-valued inclusions. In this section, we present some basic results on this method and discuss its application to general optimization problems.

\subsection{The SCD semismooth* Newton method for inclusions}\label{subsect_semismooth_inclusions}

The general SCD semismooth* Newton method aims to solve the set-valued inclusion
    \begin{equation}\label{EqInclG}
        0\in G(x) \,,
    \end{equation}
where $G:\R^n\tto \R^n$ is a set-valued mapping. To motivate the derivation of the SCD semismooth* Newton method, we first recall the well-known semismooth Newton method as introduced by Qi and Sun \cite{QiSun93} for solving equations
    \begin{equation}\label{EqNonlEqH}
        H(x)=0\,,
    \end{equation}
where $H:D\to\R^n$ is a locally Lipschitz continuous mapping defined on an open set $D\subset\R^n$. For this, let
    \begin{equation*}
        \OO_H:=\Kl{x\in\R^n\mv H\mbox{ is (Fr\'echet) differentiable at }x} \,.
    \end{equation*}
By Rademacher's Theorem, see, e.g.,\cite{RoWe98}, the set $D\setminus \OO_H$ is negligible and therefore $\OO_H$ is dense in $D$. Since $H$ is assumed to be locally Lipschitz continuous, for any $x\in D$ and any sequence $x_k\longsetto{\OO_H}x$ the sequence $\nabla H(x_k)$ is bounded, and thus has some accumulation point. From this, it follows that for every $x\in D$ the set
    \begin{equation*}
        \overline\nabla H(x):=\Kl{A\in\R^{n\times n}\mv \exists \, x_k\longsetto{\OO_H}x\mbox{ such that }A=\lim_{k\to\infty}\nabla H(x_k)} 
    \end{equation*}
is a nonempty and compact set, which is called the \emph{B-Jacobian (B-differential)} of $H$ at $x$. By taking the convex hull, $\co\overline\nabla H(x)$, we arrive at the \emph{Clarke generalized Jacobian} of $H$ at $x$. Following \cite{QiSun93}, we can introduce the concept of semismoothness.

\begin{definition}
The mapping $H : D \to \R^n$ is called \emph{semismooth} at $\xb\in D$ if the limit 
    \begin{equation*}
        \lim_{\AT{A\in\co\onabla H(\xb+th')}{h'\to h,\ t\downarrow 0}}Ah' 
    \end{equation*}
exists for every $h \in \R^n$. 
\end{definition}

Note that $H$ is semismooth at $\xb\in D$ if and only if the directional derivative
    \begin{equation*}
        H'(\xb;h):=\lim_{t\downarrow 0}\frac{H(\xb+th)-H(\xb)}{t}
    \end{equation*}
exists for all $h\in\R^n$ and there holds
    \begin{equation}\label{EqSS}
        \sup_{A\in \co\onabla H(x)}\norm{H(x)-H(\xb)-A(x-\xb)}=\oo(\norm{x-\xb}) \,.
    \end{equation}
With this, we are now able to define and analyse the semismooth Newton method.

\begin{theorem}[\cite{QiSun93}]\label{ThConvSS} Let $D\subset\R^n$ be open and let $H:D\to\R^n$ be locally Lipschitz continuous. Furthermore, assume that $\xb\in D$ is a solution of \eqref{EqNonlEqH} with 
\eqref{EqSS}, and assume that all matrices $A\in\co\onabla H(\xb)$ are nonsingular. Then, there exists a neighborhood $V$ of $\xb$ such that for every starting point $\ee x0\in V$ the \emph{semismooth Newton} iteration 
    \begin{equation}\label{EqSSNewton} 
        \ee x{k+1}=\ee xk-{\ee Ak}^{-1}H(\ee xk)  \,,\qquad \mbox{ with }\ee Ak\in\co\onabla H(\ee xk) \,,
    \end{equation}
is well defined and the resulting sequence $\ee xk$ converges superlinearly to $\xb$.
\end{theorem}

Next, we want to generalize the semismooth Newton method \eqref{EqSSNewton} to set-valued mappings $G:\R^n\tto\R^n$. In the following analysis, the \emph{graph of $G$}, defined by
    \begin{equation*}
        \gph G:=\{(x,y)\mv y\in G(x)\} \,,
    \end{equation*}
plays an important role. Note that the locally Lipschitzian mapping $H:D\to\R^n$ is Fr\'echet differentiable at $x\in D$, if and only if the tangent cone to $\gph H$ at $(x, H(x))$ is an $n$-dimensional subspace of $\R^n\times\R^n$. In this case, there holds
    \begin{equation*}
        T_{\gph H}(\xb, H(\xb))=\rge(I,\nabla H(x)):=\Kl{(p,\nabla H(\xb)p)\mv p\in\R^n} \,.
    \end{equation*}
This leads to the following definition of graphical smoothness.

\begin{definition}
The set-valued mapping $G:\R^n\tto\R^n$ is called \emph{graphically smooth of dimension $n$} at $(x,y)\in\gph G$, if the tangent cone $T_{\gph G}(x,y)$ is an $n$-dimensional subspace of $\R^n\times \R^n$. Furthermore, $\OO_G$ denotes the set of all points $(x,y)\in\gph G$ such that $G$ is graphically smooth of dimension $n$ at $(x,y)$.
\end{definition}
 
Next, we consider the metric space $\Z_n$ of $n$-dimensional subspaces of $\R^n\times\R^n$, equipped with the metric
    \begin{equation*}
        d_{\Z_n}(L_1,L_2):=\norm{P_1-P_2} \,,
    \end{equation*}
where $P_i$, $i=1,2$ denote the orthogonal projections onto the subspaces $L_i$. Since orthogonal projections are uniformly bounded linear operators, and since the limit of orthogonal projections on $n$-dimensional subspaces is again an orthogonal projection onto an $n$-dimensional subspace, the metric space $\Z_n$ is compact. This allows the construction of a \emph{subspace containing (SC) derivative}, similar to that of the B-Jacobian. 

\begin{definition}\label{DefSCDProperty} Let $G:\R^n\tto\R^n$ be a set-valued mapping. Then:
  \begin{enumerate}
    \item The \emph{subspace containing (SC) derivative} $\Sp G:\gph G\tto \Z_n$ of $G$ is defined by
    \begin{align*}
        \Sp G(x,y)&:=\Kl{L\in \Z_n\mv \exists \, (x_k,y_k)\longsetto{{\OO_G}}(x,y):\ \lim_{k\to\infty} d_\Z(L,T_{\gph G}(x_k,y_k))=0} \,.
    \end{align*}
    \item The mapping $G$ is said to have the \emph{subspace containing derivative (\SCD) property at} $(x,y)\in\gph G$, if there holds $\Sp G(x,y)\neq \emptyset$. Furthermore, we say that $G$ has the \SCD property \emph{around} $(x,y)\in\gph G$, if there is a neighborhood $W$ of $(x,y)$ such that $G$ has the \SCD property at every $(x',y')\in\gph G\cap W$. Finally, we call $G$ an \emph{\SCD mapping}, if $G$ has the \SCD property at every point of its graph.
  \end{enumerate}
\end{definition}
Since the metric space $\Z_n$ is compact, $G$ has the SCD property at $(x,y)\in \gph G$ if and only if $(x,y)\in\cl\OO_G$, and $G$ is an SCD mapping if and only if $\cl\gph G=\cl\OO_G$. Furthermore, note that for a single-valued mapping $H:D\to\R^n$ there holds
    \begin{equation*}
        \Sp H (x,H(x))\supset \{\rge(I,A)\mv A\in \onabla H(x)\} \,, \quad  \forall \, x\in D \,,
    \end{equation*}
and equality holds for all $x$ such that $H$ is Lipschitz continuous near $x$.

\begin{example}
Consider the function $H:\R\to\R$ defined by $H(x)=3x^{\frac{1}{3}}$. Then $\gph H$ is a smooth curve in $\R^2$ and thus $H$ is an SCD mapping with SC derivative
    \begin{equation*}
        \Sp H(x,H(x))=\begin{cases}\rge(1,x^{-\frac{2}{3}})\,, &\mbox{if $x\not=0$} \,,\\
        \rge(0,1)\,, &\mbox{if $x=0$} \,.
        \end{cases} 
    \end{equation*}
However, note that $\onabla H(0)=\emptyset$. This example demonstrates that there can be subspaces in the SC derivative which are not the graph of a linear mapping from $\R^n\to\R^n$.
\end{example}

In addition to the SCD semismooth property, we also need to introduce the SCD semismooth* property. For this, note that in Theorem~\ref{ThConvSS} we only assumed \eqref{EqSS} but not the directional differentiability of $H$. Furthermore, note that for locally Lipschitz continuous mappings $H$ it can be shown that condition \eqref{EqSS} is equivalent to
    \begin{equation*}
        \sup_{A\in \onabla H(x)}\Kl{\dist{\kl{x-\xb, H(x)-H(\xb)},\rge(I,A)}}=\oo\kl{\norm{(x-\xb,H(x)-H(\xb)}} \,.
    \end{equation*}
This motivates the following definition of the SCD semismooth* property.

\begin{definition}
The set-valued mapping $G:\R^n\tto\R^n$ is said to be \emph{\SCD \ssstar at} $(\xb,\yb)\in\gph G$, if $G$ has the \SCD property around $(\xb,\yb)$ and for all $\epsilon>0$ there is some $\delta>0$ such that for all $(x,y)\in\gph G\cap\B_\delta(\xb,\yb)$ and all $L\in \Sp F(x,y)$ there holds
    \begin{equation*}
        \dist{(x-\xb,y-\yb),L}\leq\epsilon\norm{(x-\xb,y-\yb)} \,.
    \end{equation*}
\end{definition}

\begin{remark} Note that the SCD \ssstar property has been defined in \cite{GfrOut22a} in a different but equivalent way by means of adjoint subspaces. The class of SCD semismooth* mappings is rather broad. For example, every SCD mapping $G$ whose graph is a closed subanalytic set is  SCD semismooth* at every point of its graph. Informally speaking, if the graph of $G$ can be described by the finite union of sets defined by smooth equalities and inequalities, then it is very likely that the mapping is SCD semismooth* at every point of its graph.
\end{remark}

After these preliminaries, we are now ready to derive the \emph{SCD semismooth* Newton method} for solving the inclusion \eqref{EqInclG}. For this, let $\xb$ be a solution of \eqref{EqInclG} and consider an iterate $\ee xk$ close to $\xb$. First, note that there is a structural difference between solving an inclusion $0\in G(x)$ and solving (nonsmooth) equations $H(x)=0$. In the latter case, it follows from the local Lipschitz continuity of $H$ that $H(\ee xk)$ is close to $0$. However, when solving inclusions $0\in G(x)$, we cannot guarantee that $0$ is close to $G(\ee xk)$. In fact, since $G$ is a set-valued mappings, it may even happen that $G(\ee xk)=\emptyset$!

Hence, before computing a new iterate $\ee{\hat x}{k+1}$, we first have to perform a so-called \emph{approximation step (AS)}: We have to compute a point $(\ee{\hat x}k,\ee{\hat y}k)\in\gph G$ satisfying
    \begin{equation}\label{EqApprStep}
        \norm{(\ee{\hat x}k-\xb,\ee{\hat y}k-0)}\leq \eta \norm{\ee xk-\xb}\,,
    \end{equation}
for some constant $\eta>0$. Such a point always exist for any $\eta>0$ by just taking $(\ee{\hat x}k,\ee{\hat y}k)=(\xb,0)$. However, this choice is not practicable. Hence, in practice we compute $(\ee{\hat x}k,\ee{\hat y}k)\in\gph G$ as an approximate projection of $(\ee xk,0)$ onto $\gph G$ satisfying
    \begin{equation*}
        \norm{(\ee{\hat x}k,\ee{\hat y}k)-(\ee xk,0)}\leq \beta\dist{(\ee xk,0),\gph G} \,,
    \end{equation*}
for some constant $\beta\geq1$. This procedure guarantees that \eqref{EqApprStep} is satisfied with $\eta=\beta+1$. Different ways to compute this approximation step in practice are discussed below.

After the approximation step, we then perform a \emph{Newton step} as follows: Observe first that for equations $H(x)=0$, the Newton iteration \eqref{EqSSNewton} can be written as 
    \begin{equation*}
        (\ee x{k+1}-\ee xk,0-H(\ee xk))\in\rge(I,A_k) \,.
    \end{equation*}
This motivates that for inclusions $0 \in G(x)$, we choose a subspace $\ee Lk\in\Sp G((\ee{\hat x}k,\ee{\hat y}k)$ and then compute the next iterate $\ee x{k+1}$ of the SCD semismooth* Newton method via the requirement
    \begin{equation}\label{EqSS*Newton}
        (\ee x{k+1}-\ee{\hat x}k, -\ee{\hat y}k)\in\ee Lk \,.
    \end{equation}
Now, the question arises whether the inclusion \eqref{EqSS*Newton} has a unique solution. This is the case under a regularity condition, for the statement of which we need the following:

\begin{definition}
Let $\Z_n^{\rm reg}$ denote the collection of all subspaces $L\in\Z_n$ such that
    \begin{equation}\label{EqSCDReg_L}
        (y^*,0)\in L\ \Longrightarrow\ y^*=0 \,.
    \end{equation}
\end{definition}

Next, note that one can show that for a given subspace $L\in\Z_n^{\rm reg}$ there exists a unique $n\times n$ matrix $C_L$ such that $L=\rge(C_L,I)$. Thus, given any two $n\times n$ matrices $(X,Y)$ such that $L=\rge(X,Y)\in\Z_n^{\rm reg}$, it follows that $Y$ is nonsingular and $C_L=XY^{-1}$. Hence, if $\ee Lk=\rge(\ee Xk,\ee Yk)\in\Z_n^{\rm reg}$, then the inclusion \eqref{EqSCDReg_L} has the unique solution
    \begin{equation*}
        \ee x{k+1}=\ee{\hat x}k -C_{\ee Lk}\ee{\hat y}k \,,
    \end{equation*}
or equivalently, using $C_L=XY^{-1}$,
    \begin{equation}\label{EqSS*NewtonXY}
        \ee x{k+1}=\ee{\hat x}k +\ee Xk\ee pk \,, \qquad \text{where} \qquad \mbox{$\ee pk$ \, solves \,\,} \ee Yk \ee pk =-\ee{\hat y}k \,.
    \end{equation}

Summarizing these steps, we now arrive at the following algorithm.

\begin{algorithm}[\emph{SCD semismooth* Newton method} for solving the inclusion \eqref{EqInclG}]\label{algo_SCD_sss_inclusions}
Given $G: \R^n \tto \R^n$, $\ee x0 \in \R^n$, and $\eta > 0$. Repeat until convergence:
    \begin{enumerate}
        \item{\textbf{Approximation step (AS):}} Compute $(\ee {\hat x}k,\ee{\hat y}k)\in\gph G$ satisfying \eqref{EqApprStep}
        \item{\textbf{Newton step (NS):}} Select two $n\times n$ matrices  $\ee Xk, \ee Yk$ with 
            \begin{equation*}
                \rge(\ee {X}k,\ee Yk)\in\Sp G((\ee {\hat x}k,\ee{\hat y}k))\cap \Z_n^{\rm reg} \,,
            \end{equation*} 
        and use these matrices to compute the next iterate $\ee x{k+1}$ via \eqref{EqSS*NewtonXY}.
    \end{enumerate}
\end{algorithm}

Note that for an implementation of the SCD semismooth* Newton method given in Algorithm~\ref{algo_SCD_sss_inclusions}, only \textbf{one} subspace $\ee Lk=\rge(\ee Xk,\ee Yk)\in \Sp G((\ee {\hat x}k,\ee{\hat y}k))\cap \Z_n^{\rm reg}$ is required per iteration. Concerning convergence of the method, we have the following:

\begin{theorem}\label{thm_conv_inclusion}
Let $\xb\in\R^n$ be a solution of the inclusion \eqref{EqInclG}. Furthermore, let $G:\R^n\tto\R^n$ be SCD semismooth* at $(\xb,0)$ and assume that the regularity condition
    \begin{equation}\label{EqSCDRegSol}
        \Sp G(\xb,0)\subseteq \Z_n^{\rm reg}
    \end{equation}
holds. Then, for every $\eta>0$ in \eqref{EqApprStep}, there is a neighborhood $U$ of $\xb$ such that for every starting point $\ee x0\in U$ the SCD semismooth* Newton method given by Algorithm~\ref{algo_SCD_sss_inclusions} is well defined, and the iterates $\ee xk$ converge superlinearly to $\xb$.
\end{theorem}
\begin{proof}
The proof of this theorem can be found in \cite{GfrOut22a}. However, note that in this work the dual formulation via the adjoint subspace containing (SC) derivative is used.
\end{proof}

Note that the approximation step in the SCD semismooth* Newton method is indispensable. For general mappings $G$, we do not know how to perform this approximation step. However, if, e.g., $G$ is a maximally monotone mapping, one can choose
    \begin{equation*}
        \ee{\hat x}k=(I+\lambda G)^{-1}(\ee xk)\,, \qquad \text{and} \qquad 
        \ee {\hat y}k=(\ee xk-\ee {\hat x}k)/\lambda \,,
    \end{equation*}
for any $\lambda>0$. Since the resolvent $(I+\lambda G)^{-1}$ is non-expansive, cf.~\cite[Theorem 12.12]{RoWe98}, and since $(I+\lambda G)^{-1}(\xb)=\xb$, we obtain that
    \begin{equation*}
        \norm{\ee {\hat  x}k-\xb}\leq \norm{\ee xk-\xb}\,, 
        \qquad \text{and} \qquad 
        \norm{\ee{\hat y}k}\leq\norm{\ee xk-\xb}/\lambda \,.
    \end{equation*}
Hence, condition \eqref{EqApprStep} holds with $\eta=\sqrt{1+\lambda^{-2}}$. Similarly, if $G=G_1+G_2$, where $G_1$ is a single-valued Lipschitz mapping and $G_2$ is maximally monotone, then $(\ee {\hat x}k, \ee {\hat y}k)$ satisfying \eqref{EqApprStep} can be computed via forward-backward splitting as
    \begin{equation*}
        \ee {\hat x}k=(I+\lambda G_2)^{-1}(I-\lambda G_1)(\ee xk)\,,
        \qquad\ee{\hat y}k=(\ee xk-\ee {\hat x}k)/\lambda- (G_1(\ee xk)-G_1(\ee {\hat x}k)) \,.
    \end{equation*}

\subsection{The SCD semismooth* Newton method in optimization}\label{subsect_semismooth_optimization}

After considering the case of general inclusions $0 \in G(x)$ in the previous section, we now want to apply the SCD semismooth* Newton method to the particular inclusion
    \begin{equation}\label{EqInclsubphi}
        0\in\partial\varphi(x) \,,
    \end{equation}
where $\varphi:\R^n\to\oR$ is proper. For this, we first need the following result.

\begin{proposition}[{\cite[Proposition~3.26]{GfrOut22a}}]\label{PropProxRegularQ}
Let $\varphi:\R^n\to\oR$ be prox-regular and subdifferentially continuous at $\xb$ for $\xb^*\in\partial \varphi(\xb)$. Then there holds:
    \begin{enumerate}
        \item The subdifferential mapping $\partial \varphi$ has the \SCD property around $(\xb,\xba)$ \,.
        \item There exists a $\lambda>0$ such that for every $(x,x^*)\in\gph\partial \varphi$ sufficiently close to $(\xb,\xba)$ and for every $L\in\Sp\partial \varphi(x,x^*)$ there is a symmetric positive semidefinite $n\times n$ matrix $B$ such that 
            \begin{equation*}
                L=\rge\kl{B,\frac 1\lambda(I- B)} \,. 
            \end{equation*}
    \end{enumerate}
\end{proposition}

Hence, the following additional assumption on $\varphi$ ensures that $\partial\varphi$ is an SCD mapping.

\begin{assumption}\label{AssProx}
The set $\operatorname{proxreg} \varphi$ is dense in $\gph\partial\varphi$, where
    \begin{equation*}
        \operatorname{proxreg} \varphi:=\left\{(x,x^*)\in\gph \partial\varphi \, \Big\vert \, \parbox{7cm}{ $\varphi$ is prox-regular and subdifferentially continuous at $x$ for $x^*$}\right\} \,.
    \end{equation*}
\end{assumption}

\begin{lemma}[{\cite[Lemma~3.1]{Gfr24}}] 
If Assumption~\ref{AssProx} holds, then $\partial \varphi$ is an SCD mapping.
\end{lemma}

In the remainder of this section, we always assume that Assumption~\ref{AssProx} is satisfied without further mention. Note that Assumption~\ref{AssProx} is rather weak: Even though one can find proper lsc functions $\varphi$ violating it, these examples are purely academic and do not have any practical relevance.

Proposition \ref{PropProxRegularQ} provides a certain basis representation of the subspaces $L\in\Sp\partial \varphi(x,x^*)$ via the positive semidefinite matrices $B$. However, this basis representation is not very convenient to work with and therefore we consider the following alternative.

\begin{definition}\label{DefZnP}
Let $\ZnP$ denote the collection of all subspaces $L\in\Z_n$ such that there are symmetric $n\times n$ matrices $P$ and $W$ with the following properties:
    \begin{enumerate}
        \item[(i)] $L=\rge(P,W)$ \,,
        \item[(ii)] $P^2=P$, i.e., $P$ represents the orthogonal projection onto some subspace of $\R^n$,
        \item[(iii)] $W(I-P)=I-P$.
  \end{enumerate}
\end{definition}

Next, note that if $W,P$ are symmetric matrices fulfilling the conditions (ii) and (iii) in the above definition, then it follows that 
    \begin{gather}
        \label{EqPW1}
        (I-P)W=\big(W(I-P)\big)^T=(I-P)^T=I-P=W(I-P) \,,
        \\
        \label{EqPW2}
        PW=WP \,,
        \\
        \label{EqPW3}
        W=(P+I-P)W(P+I-P)=PWP+I-P \,,
    \end{gather}
These properties are used below as well as in the proof of the following proposition.
    
\begin{proposition}[{\cite[Proposition~3.6]{Gfr24}}]\label{PropPW_Basis}
For every $(x,x^*)\in\operatorname{proxreg}\varphi$ there holds 
    \begin{equation*}
        \emptyset\not=\Sp\partial\varphi(x,x^*)\subseteq \ZnP \,.
    \end{equation*}
Furthermore, for every $(x,x^*)\in\gph \varphi$ one has
    \begin{equation}\label{EqSpPW}
        \Sp_{P,W}\partial \varphi(x,x^*):=\Sp\partial\varphi(x,x^*)\cap \ZnP\neq\emptyset \,.
    \end{equation}
\end{proposition}

Note that if $\varphi$ is twice Fr\'echet differentiable at $\xb$, then $\partial\varphi(\xb)=\{\nabla \varphi(\xb)\}$ and
    \begin{equation*}
        \rge(I,\nabla^2\varphi(\xb))\in\Sp\partial\varphi(\xb,\nabla\xb) \,.
    \end{equation*} 
Furthermore, if $\varphi$ is twice continuously differentiable at $\xb$ then we even have 
    \begin{equation*}
        \Sp\partial\varphi(\xb,\nabla\xb)=\Kl{\rge(I,\nabla^2\varphi(\xb))} \,.
    \end{equation*}
For examples with nonsmooth functionals $\varphi$, please refer to Appendix~\ref{appendix_examples}, where we consider the case of $\ell_p$-norms with $0 \leq p \leq 1$ for later use in Section~\ref{subsect_appl_sparsity}.

In the SCD semismooth* Newton method given in Algorithm~\ref{algo_SCD_sss_inclusions}, the use of subspaces belonging to $\Sp_{P,W}\partial \varphi(x,x^*)$ has some advantages when computing the Newton direction. In particular, assume that we are given some subspace
    \begin{equation*}
        L=\rge(P,W)\in \Sp_{P,W}\partial \varphi(x,x^*) \overset{\eqref{EqSpPW}}{=} 
        \Sp\partial\varphi(x,x^*)\cap \ZnP \,,
    \end{equation*}
and that, for solving \eqref{EqSS*Newton}, we want to find $\bar s$ with $(\bar s,-x^*)\in L$. This is equivalent to
    \begin{equation*}
        \bar s=Pp \,, 
        \qquad \text{for some} \quad p \quad  \text{with} \,\, Wp=-x^* \,.
    \end{equation*}
Now, by using the property \eqref{EqPW2}, we obtain that
    \begin{equation*}
        -Px^*=PWp = Wp =W\bar s \,.
    \end{equation*}
Conversely, if $\bar s$ satisfies $W\bar s=-Px^*$, then due to \eqref{EqPW3} we obtain 
    \begin{equation*}
        (I-P)\bar s=-Px^*-PWP\bar s\in\rge P \,,
    \end{equation*}
which implies that $\bar s\in\rge P$. Now, setting $p=\bar s-(I-P)x^*$ we obtain $\bar s=Pp$ and $Wp=W\bar s+(I-P)p=-x^*$. Together with \cite[Lemma 3.13]{Gfr24}, this yields:

\begin{lemma}\label{LemStatPoint}
  Let $(x,x^*)\in\gph \partial \varphi$, $L=\rge(P,W)\in \Sp_{P,W}\partial\varphi(x,x^*)$ and $\bar s\in\R^n$ be given. Then the following three statements are equivalent:
  \begin{enumerate}
    \item[(i)] $(\bar s,-x^*)\in L$, 
    \qquad  
    (ii) $W\bar s=-Px^*$,
    \item[(iii)] $\bar s$ is a stationary point for the quadratic program
        \begin{equation}\label{EqQP}
            \min_s \frac 12s^TWs+ \skalp{x^*,s} \,, \quad\mbox{ subject to }\quad s\in\rge P \,,
        \end{equation}
    i.e., $0\in \partial q_{P,W}(\bar s)$, where
        \begin{equation*}
            q_{P,W}(s):=\frac 12s^TWs+\skalp{z^*,s}+\delta_{\rge P}(s)
            =\begin{cases}
                \frac 12s^TWs+\skalp{z^*,s}\,, &\mbox{if $s\in\rge P$} \,,
                \\
                \infty \,, &\mbox{else} \,, \end{cases} 
        \end{equation*}
    or, equivalently, there exists some multiplier $\bar s^*$ such that $W\bar s+x^*+(I-P)\bar s^*=0$.
  \end{enumerate}
\end{lemma}

Using this result, we obtain the following variant of Algorithm~\ref{algo_SCD_sss_inclusions} for solving  \eqref{EqInclsubphi}:

\begin{algorithm}[\emph{SCD semismooth* Newton method} for solving the inclusion \eqref{EqInclsubphi}]\label{algo_SCD_sss_subphi}
Given $\varphi: \R^n \to (-\infty,\infty]$, $\ee x0 \in \R^n$, and $\eta > 0$. Repeat until convergence:
    \begin{enumerate}
        \item{\textbf{Approximation step (AS):}} Compute $(\ee {\hat x}k,\ee{\hat y}k)\in\gph \partial \varphi$ satisfying \eqref{EqApprStep}.
        \item{\textbf{Newton step (NS):}} Select two $n\times n$ matrices  $\ee Pk, \ee Wk$ with 
            \begin{equation*}
                \rge(\ee Pk,\ee Wk)\in \Sp_{P,W}\partial \varphi(\ee {\hat x}k,\ee{\hat y}k) \overset{\eqref{EqSpPW}}{=} 
                \Sp\partial\varphi(\ee {\hat x}k,\ee{\hat y}k)\cap \ZnP \,,
                \end{equation*}
        compute a stationary point $\ee{\bar s}{k}$ of the quadratic program
        \begin{equation}\label{helper_SCD_quadratic}
            \min_s \frac 12s^T \ee Wk s+ \skalp{\ee{\hat y}k,s} \,, \quad\mbox{ subject to }\quad s\in\rge \ee Pk \,,
        \end{equation}
        and define the next iterate $\ee x{k+1} := \ee{\hat x}k + \ee{\bar s}{k}$.
    \end{enumerate}
\end{algorithm}

For ways to solve the approximation step in Algorithm~\ref{algo_SCD_sss_subphi}, see, e.g., Section~\ref{subsect_globalization} or the last paragraph of Section~\ref{subsect_semismooth_inclusions}. Furthermore, combining Theorem~\ref{thm_convergence_subphi} with the above observations, in particular Lemma~\ref{LemStatPoint}, we obtain the following convergence results.

\begin{theorem}\label{thm_convergence_subphi}
Let $\xb\in\R^n$ be a solution of the inclusion \eqref{EqInclsubphi}, let $\varphi: \R^n \to (-\infty,\infty]$ be proper, and let Assumption~\ref{AssProx} hold. Furthermore, let $\partial \varphi$ be SCD semismooth* at $(\xb,0)$, and assume that the regularity condition $\Sp \partial \varphi(\xb,0) \subset \Z_n^{\rm reg}$ holds. Then, for every $\eta>0$ in \eqref{EqApprStep}, there is a neighborhood $U$ of $\xb$ such that for every starting point $\ee x0\in U$ the SCD semismooth* Newton method given by Algorithm~\ref{algo_SCD_sss_subphi} is well defined, and the iterates $\ee xk$ converge superlinearly to $\xb$.
\end{theorem}

\section{SCD semismooth* Newton and Tikhonov}\label{sect_semismooth_Tikhonov}

In Section~\ref{sect_semismooth} we considered the SCD semismooth* Newton method for solving general inclusions $0 \in G(x)$, as well as particular inclusions of the form  $0\in\partial\varphi(x)$. Now, recall that our initial goal was to minimize Tikhonov functionals of the form \eqref{Tikhonov_finite}, i.e.,
    \begin{equation*}
        \Tad(x) = S\kl{F(x),\yd} + \alpha R(x) \,.
    \end{equation*}
This problem can be rewritten as the compund minimization problem 
    \begin{equation}\label{EqOptProbl}
        \min \varphi(x):=f(x)+g(x) \,,
    \end{equation}
where the functionals $f$ and $g$ are defined by
    \begin{equation*}
        f(x) := S\kl{F(x),\yd} \,,
        \qquad \text{and} \qquad
        g(x) := \alpha R(x) \,.
    \end{equation*}
To find a stationary point of the functional $\varphi$, we have to solve the inclusion $0 \in \partial \varphi$, cf.~\eqref{first_order_optimality_nonsmooth_finite}. Hence, we now adapt the SCD semismooth* Newton method to composite optimization problems of the form \eqref{EqOptProbl}. Throughout this section, we make the following

\begin{assumption}\label{assumption_fg}
The function $\varphi : \R^n \to (-\infty,\infty]$ has the form $\varphi = f + g$, where $f:\R^n\to\R$ is twice continuously differentiable and $g:\R^n\to\oR$ is a proper lsc prox-bounded function.     
\end{assumption}

Furthermore, for the practical implementation of the resulting SCD semismooth* Newton method it is required that for every $\lambda>0$ sufficiently small and every $x\in\R^n$, one element of the proximal mapping $\prox_{\lambda g}(x)$ can be  calculated.

\subsection{SCD semismooth* Newton for Tikhonov functionals}

First, we start by considering the following sum rule for the composite functional $\varphi$.

\begin{lemma}[{\cite[Lemma 3.16]{Gfr24}}]\label{LemSumRul}Assume that $\varphi=f+g$, where $f$ is twice continuously differentiable at $\xb$ and $g:\R^n\to\oR$ is a proper lsc function. Given $\xb\in\partial g(\xb)$, for every subspace $\rge(P,W)\in\Sp_{P,W}\partial g(\xb,\xba)$ there holds
    \begin{equation*}
        rge(P,P\nabla^2 f(\xb)P +W)\in \Sp_{P,W}\partial \varphi(\xb,\nabla f(\xb)+\xba) \,.
    \end{equation*}
\end{lemma}

This result is the basis of the following adapted SCD semismooth* Newton method:

\begin{algorithm}[\emph{SCD semismooth* Newton method} for solving the inclusion \eqref{EqInclsubphi}]\label{algo_SCD_sss_fpg}
Let Assumption~\ref{assumption_fg} hold, and let $\ee x0 \in \R^n$ and $\eta > 0$. Repeat until convergence:
    \begin{enumerate}
        \item{\textbf{Approximation step (AS):}} Compute $(\ee {\hat x}k,\ee{\hat y}k)\in\gph \partial \varphi$ satisfying \eqref{EqApprStep}. 
        \item{\textbf{Newton step (NS):}} Select two $n\times n$ matrices  $\ee Pk, \ee Wk$ with 
            \begin{equation*}
                \rge(\ee Pk,\ee Wk)\in \Sp_{P,W}\partial g(\ee {\hat x}k,\ee{\hat y}k) \overset{\eqref{EqSpPW}}{=} 
                \Sp\partial g(\ee {\hat x}k,\ee{\hat y}k)\cap \ZnP \,,
                \end{equation*}
        compute a stationary point $\ee{\bar s}{k}$ of the quadratic program
        \begin{equation*}
            \min_s \frac 12s^T (\ee Pk \nabla^2 f(\xb)\ee Pk + \ee Wk) s+ \skalp{\ee{\hat y}k,s} \,, \quad\mbox{ subject to }\quad s\in\rge \ee Pk \,,
        \end{equation*}
        and define the next iterate $\ee x{k+1} := \ee{\hat x}k + \ee{\bar s}{k}$.
    \end{enumerate}
\end{algorithm}

\begin{remark}
Concerning the approximation step in Algorithm~\ref{algo_SCD_sss_fpg}, note that if the functional $g$ is proper, lsc, and convex, then this is achieved by (for any $\lambda >0$) the choice
    \begin{equation*}
       \ee {\hat x}k=(I+\lambda \partial g)^{-1}(I-\lambda \nabla f)(\ee xk)\,,
       \qquad\ee{\hat y}k=(\ee xk-\ee {\hat x}k)/\lambda- (\nabla f(\ee xk)-\nabla f(\ee {\hat x}k)) \,.
    \end{equation*}
\end{remark}

For this modified algorithm, we have the following result in analogy to Theorem~\ref{thm_convergence_subphi}.

\begin{theorem}\label{thm_convergence_fpg}
Let $\xb\in\R^n$ be a solution of the inclusion \eqref{EqInclsubphi} and let Assumption~\ref{AssProx} and Assumption~\ref{assumption_fg} hold. Furthermore, let $\partial g$ be SCD semismooth* at $(\xb,-\nabla f(\xb))$, and assume that the regularity condition $\Sp \partial \varphi(\xb,0) \subset \Z_n^{\rm reg}$ holds. Then, for every $\eta>0$ in \eqref{EqApprStep}, there is a neighborhood $U$ of $\xb$ such that for every starting point $\ee x0\in U$ the SCD semismooth* Newton method given by Algorithm~\ref{algo_SCD_sss_fpg} is well defined, and the iterates $\ee xk$ converge superlinearly to $\xb$.
\end{theorem}

\subsection{Globalizing the SCD semismooth* Newton method}\label{subsect_globalization}

In this section, we discuss a hybrid of the SCD semimooth* Newton method and the well-known \emph{proximal gradient method (PGM)} for solving \eqref{EqOptProbl}. In the resulting method, PGM ensures global convergence of the algorithm, while the SCD semimooth* Newton method yields fast convergence close to a solution. Furthermore, the PGM acts as an approximation step for the SCD semismooth* Newton method.

First, for a given $\lambda>0$, we define the function $\psi_\lambda^{\rm FB}:\R^n\times\R^n\to \oR$ by
    \begin{equation*}
    \begin{split}
        \psi_\lambda^{\rm FB}(x,z)&:=\ell_f(x,z)+\frac1{2\lambda}\norm{z-x}^2+g(z) \,,
        \\
        \ell_f(x,z)&:=f(x)+\skalp{\nabla f(x),z-x} \,.
    \end{split}
    \end{equation*}
Furthermore, we define the mapping $T_\lambda:\R^n\tto\R^n$ and the function $\varphi_\lambda^{\rm FB}: \R^n\to\R$ by
    \begin{align*}
        T_\lambda(x):=\argmin_z\psi_\lambda^{\rm FB}(x,z) \,,
        \qquad \text{and} \qquad 
        \varphi_\lambda^{\rm FB}(x):=\min_z \psi_\lambda^{\rm FB}(x,z) \,.
    \end{align*}
Clearly, for any $z\in T_\lambda(x)$, we have the first-order optimality condition
    \begin{equation}\label{EqFO}
        0\in\widehat\partial_z \psi_\lambda^{\rm FB}(x,z)=\nabla f(x)+\frac 1\lambda(z-x)+\widehat\partial g(z) \,.
    \end{equation}
Hence, if $x\in T_\lambda(x)$ then $0\in\nabla f(x)+\widehat\partial g(x)=\widehat\partial \varphi(x)$, i.e., $x$ is a stationary point for $\varphi$. The function $\varphi_\lambda^{\rm FB}$ is called the \emph{forward-backward envelope}, and has been introduced in \cite{PaBe13} under the name \emph{composite Moreau envelope}. It has been successfully used in a series of papers for the solution of nonsmooth optimization problems, cf.~\cite{StThPa17, ThStPa18, StThPa19, MaTh22}.

We now state some basic properties of $T_\lambda$ and $\varphi_\lambda^{\rm FB}$. First of all, note that
    \begin{equation*}
        T_\lambda(x)={\rm prox}_{\lambda g}(x-\lambda\nabla f(x)) \,,
    \end{equation*}
and denote by $\lambda_g$ the threshold of prox-boundedness of $g$ as defined in Section~\ref{sect_background_variational}. 

\begin{lemma}[{\cite[Lemma 2.2]{MaTh22}}]\label{LemBasFBE1}Let $\lambda\in(0,\lambda_g)$. Then the following hold:
\begin{enumerate}
    \item[(i)] $\varphi_\lambda^{\rm FB}$ is everywhere finite and locally Lipschitz continuous.
    \item[(ii)] $\varphi_\lambda^{\rm FB}(x)\leq \varphi(x)$ for all $x\in\R^n$, and equality holds if and only if $x\in T_\lambda(x)$.
    \item[(iii)] If $\tilde x\in T_\lambda(x)$ and $f(\tilde x)\leq \ell(x,\tilde x)+\frac L2\norm{\tilde x-x}^2$ then
        \begin{align*}
            \varphi(\tilde x)\leq \varphi_\lambda^{\rm FB}(x)-\frac{1-\lambda L}{2\lambda}\norm{\tilde x-x}^2 \,.
        \end{align*}
\end{enumerate}
\end{lemma}

We now sketch the \emph{Basic Globalized SemiSmooth* Newton (BasGSSN) method}, a globally convergent method for solving \eqref{EqOptProbl} which was first derived in \cite[Algorithm~1]{Gfr24}. 

\begin{algorithm}[\emph{Basic Globalized SemiSmooth* Newton (BasGSSN) method} for solving \eqref{EqOptProbl}]\label{algo_BasGSSN}
Let Assumption~\ref{assumption_fg} hold, and let $\ee x0 \in \R^n$ and $\eta > 0$. Furthermore, fix some parameters $\alpha,\beta\in(0,1)$, $\sigma\in(0,\frac 12)$, $\bar \rho>0$, and $\bar\lambda\in(0,\lambda_g)$.
Repeat until convergence:
    \begin{enumerate}
        \item Select a search direction $\ee sk\in\R^n$ with $\norm{\ee sk}\leq\bar\rho$. (For convergence, this can be any direction, but for efficiency later the SCD semismooth* Newton direction is used.)
        \item Determine a step size $\ee \tau k\in\{2^{-i}\mv i\in\N_0\}$, and $\ee\lambda{k+1}\in \{2^i\ee \lambda k \mv i\in\mathbb{Z}\}$, such that the new iterates $\ee x{k+1}=\ee zk+\ee\tau k\ee sk$ and $\ee z{k+1}\in T_{\ee\lambda {k+1}}(\ee x{k+1})$ fulfill
            \begin{align}\label{EqBasGSSN1a}
                &\phiFB{k+1}\leq\phiFB{k}-\beta(1-\alpha)\ee \eta k \,,
                \\
                \label{EqBasGSSN1b}
                &f(\ee z{k+1})\leq \ell_f(\ee x{k+1},\ee z{k+1})+\alpha\ee \eta{k+1} \,,
                \\
                & \qquad \qquad \qquad \qquad \text{with} \quad \ee \eta{k+1}:=\frac 1{2\ee\lambda{k+1}}\norm{\ee z{k+1}-\ee x{k+1}}^2 \,,
            \end{align}
        and in addition one of the following three conditions is satisfied:
        \begin{align}\label{EqBasGSSN2a}
            &f(\ee z{k+1})\geq \ell_f(\ee x{k+1},\ee z{k+1})+\sigma\alpha\ee \eta{k+1} \,,
            \\
            \label{EqBasGSSN2b}&\ee \lambda {k+1}>\frac{\bar\lambda}2 \,,
            \\
            \label{EqBasGSSN2c}
            &f(\tilde z) \geq \ell_f(\ee x{k+1}, \tilde z)+\alpha\frac 1{4\ee\lambda{k+1}}\norm{\tilde z-\ee x{k+1}}^2\,, \text{ where }\tilde z\in T_{2\ee \lambda {k+1}}(\ee x{k+1}) \,.
    \end{align}
    \end{enumerate}
\end{algorithm}

\begin{remark}
Step~2 of Algorithm~\ref{algo_BasGSSN} can be performed as follows: We start with $\ee \tau k=1$ and $\ee \lambda{k+1}=\ee\lambda k$. If \eqref{EqBasGSSN1a} is violated, then we divide $\ee \tau k$ by $2$ and recompute $\ee x{k+1}$ and $\ee z{k+1}$. Otherwise, if \eqref{EqBasGSSN1a} is fulfilled and \eqref{EqBasGSSN1b} is violated, then we divide $\ee\lambda{k+1}$ by $2$ and recompute $\ee z{k+1}$. We repeat this procedure until both \eqref{EqBasGSSN1a} and \eqref{EqBasGSSN1b} are satisfied. One can show that this occurs after finitely many repetitions. Now it may happen that the value of $\ee\lambda{k+1}$ is too small, i.e., all three conditions \eqref{EqBasGSSN2a}-\eqref{EqBasGSSN2c} are violated. If this is the case, we multiply $\ee\lambda{k+1}$ by $2$ and recompute $\ee z{k+1}$ until one of the conditions \eqref{EqBasGSSN2a}-\eqref{EqBasGSSN2c} is satisfied. Due to \eqref{EqBasGSSN2c}, condition \eqref{EqBasGSSN1b} remains valid, and, since for any $x$ and any $0<\lambda_1<\lambda_2<\lambda_g$ we have $\varphi_{\lambda_2}^{\rm FB}(x)\leq\varphi_{\lambda_1}^{\rm FB}(x)$, the inequality \eqref{EqBasGSSN1a} also remains valid. For further details, we refer to \cite{Gfr24}.    
\end{remark}

The motivation behind the BasGSSN method defined by Algorithm~\ref{algo_BasGSSN} is as follows: Due to Lemma~\ref{LemBasFBE1}, condition \eqref{EqBasGSSN1b} ensures that there holds
    \begin{equation*}
        \varphi_\lambda^{\rm FB}(\ee zk)\leq \varphi(\ee zk)\leq \phiFB{k}-(1-\alpha)\ee \eta k\, \qquad \forall \lambda\in(0,\lambda_g) \,.
    \end{equation*}
In the pure PGM method, one simply takes $\ee x{k+1}=\ee z{k+1}$ (which corresponds to the choice $\ee sk=0)$, and adjusts $\ee \lambda{k+1}$ such that \eqref{EqBasGSSN1b} holds. This results in the decrease
    \begin{equation*}
        \phiFB{k+1}\leq \phiFB{k}-(1-\alpha)\ee \eta k \,,
    \end{equation*}
in the forward-backward envelope $\phiFB{k}$. In contrast, in BasGSSN we potentially sacrifice a part of this decrease in order to perform a line search along the direction $\ee sk$. However, due to \eqref{EqBasGSSN1a} we still have a descent in the forward-backward envelope.

\begin{theorem}[{\cite[Theorem 4.4]{Gfr24}}]\label{ThConvBasGSSN}
Consider the iterates of the BasGSSN method described in Algorithm~\ref{algo_BasGSSN} above. Then either $\varphi(\ee zk)\to-\infty$ or the following holds:
\begin{enumerate}
    \item[(i)] The sequence $\phiFB k$ is strictly  decreasing and converges to some finite value $\bar\varphi \geq \inf \varphi$.
    \item[(ii)] Every accumulation point of the sequence $\ee xk$ is also an accumulation point of the sequence $\ee zk$ and vice versa.
    \item[(iii)] Every accumulation point $\xb$ of the sequence $\ee xk$ satisfies $\xb\in T_{\underline \lambda}(\xb)$ for some $\underline\lambda\in (0,\lb)$. In particular, $\xb$ is a stationary point for $\varphi$ fulfilling $0\in\hat \partial\varphi(\xb)$.
\end{enumerate}
\end{theorem}

Next, we combine the BasGSSN method given in Algoritm~\ref{algo_BasGSSN} with the SCD semismooth* Newton method via a suitable choice of the search direction $\ee sk $. For this, note that by the definition of $T_\lambda(\cdot)$  and the first-order optimality condition \eqref{EqFO} we have
    \begin{equation*}
        \ee{z_g^*}k:=-\nabla f(\ee xk)- \frac 1{\ee \lambda k}(\ee zk-\ee xk)\in  \widehat\partial g(\ee zk) \,,
    \end{equation*}
and consequently
    \begin{equation}\label{Eqz*k}
        \ee{z^*}k:=\nabla f(\ee zk)+\ee{z_g^*}k\in\widehat\partial \varphi(\ee zk)\,.
    \end{equation}
It can be shown that the choice $(\ee zk,\ee {z^*}k)$ fulfills the requirement \eqref{EqApprStep} of the approximation step for $\ee xk$. Next, given some subspace $\rge(P_k,W_k)\in \Sp_{P,W}\partial g(\ee zk,\ee{z_g^*}k)$, we can compute
a subspace $\rge(\ee Pk,\ee Wk)\in\Sp_{P,W}\partial\varphi(\ee zk, \ee{z^*}k)$ by means of Lemma~\ref{LemSumRul}. Now, combining the definition of the SCD semismooth* Newton direction (via the quadritic problem \eqref{helper_SCD_quadratic}) with the
requirement $\norm{\ee sk}\leq\bar \rho$ and taking into account Lemma \ref{LemStatPoint}, our search direction $\ee sk$ should ideally satisfy
    \begin{equation*}
        \ee Wk\ee sk=-\ee Pk\ee{z^*}k \,, \qquad \norm{\ee sk}\leq\bar\rho \,.
    \end{equation*}
In general, we do not know whether this system has a solution, and even if a solution exists, it might be very time consuming to compute for large scale problems. Therefore, we consider an approximate solution $\ee sk$ satisfying 
    \begin{equation*}
        \norm{\ee sk}\leq \bar\rho \,,
        \qquad \text{and} \qquad 
        \ee sk\in \rge \ee Pk \,.
    \end{equation*}
Denoting the relative residual of this direction by
    \begin{equation}\label{EqRelRes}
        \ee\xi k:=\frac{\norm{\ee Wk \ee sk+\ee Pk\ee{z^*}k}}{\norm{\ee {z^*}k}} \,,
    \end{equation}
we are now able to state the following convergence results.

\begin{theorem}[{\cite[Theorem~5.1]{Gfr24}}]\label{ThSuperLinConv} Consider the iterates generated by the BasGSSN method defined in Algorithm~\ref{algo_BasGSSN} and assume that $\ee sk\in \rge\ee Pk$ for all $k$. Furthermore, assume that $\varphi(\ee zk)$ is bounded from below and that the sequence $\ee xk$ has a limit point $\xb$ which is a local minimizer for $\varphi$. (Note that due to Theorem~\ref{ThConvBasGSSN}, every accumulation point of $\ee xk$ is a stationary point for $\varphi$.) Moreover, suppose that $\partial\varphi$ is both \SCD regular and \SCD \ssstar at $(\xb,0)$.
If  for every subsequence $\ee x{k_i}\to  \xb$ the corresponding subsequence of relative residuals $\ee \xi {k_i}$ converges to $0$,
then $\ee xk$ converges superlinearly to $\xb$, and for all $k$ sufficiently large there holds $\ee x{k+1}=\ee zk+\ee sk$.
\end{theorem}

We now discuss how to compute an appropriate search direction $\ee sk$ in step~1 of Algorithm~\ref{algo_BasGSSN}. First of all, note that for improving the efficiency of the method, we require that $\norm{\ee sk}\leq\ee\rho k$, where $\ee \rho k$ is updated by
    \begin{equation*}
        \ee \rho{k+1}=\begin{cases}\max\{\frac 14\ee\rho k,\underline\rho\} \,, &\mbox{if $\ee \tau k<\frac 14$} \,,\\
        \min\{2\ee\rho k,\bar\rho\} \,, &\mbox{if $\ee \tau k=1$ and $\norm{\ee sk}=\ee\rho k$} \,,\\
        \ee \rho k \,, &\mbox{else} \,, 
        \end{cases}
    \end{equation*}
for some $0<\underline\rho<\bar\rho$. This rule helps to reduce the number of backtracking steps to compute $\ee\tau k$ in step~2 of Algorithm~\ref{algo_BasGSSN}. Next, we distinguish between two cases:

\begin{enumerate}
    \item If solving linear systems with $\ee Wk$  is not very time consuming, we can compute $\ee sk$ as follows: If $\ee Wk$ is nonsingular, compute $\ee sk$ as the projection of $-{\ee Wk}^{-1}\ee Pk\ee{z^*}k$ onto $\B_{\ee \rho k}(0)$. Otherwise, take $\ee sk:=-\frac{\ee\rho k}{\norm{\ee Pk\ee{z^*}k}}\ee Pk\ee{z^*}k$.
    \item If we cannot use a direct solver for solving the linear system, we have to use some iterative method. Motivated by Lemma~\ref{LemStatPoint}, we try to find an approximate solution of the subproblem
        \begin{equation}\label{EqTRSubProbl_s}
            \min\frac 12 s^T\ee Wk s+ \skalp{\ee {z^*}k,s}\quad \mbox{subject to}\quad s\in\rge \ee Pk,\ \norm{s}\leq \ee\rho k \,.
        \end{equation}
    To solve this subproblem, consider an $n\times m$ matrix $\ee Zk$, where $m=\dim\rge \ee Pk$, whose columns form a basis for $\rge \ee Pk$ and thus $\ee Pk=\ee Zk\big({\ee Zk}^T\ee Zk\big)^{-1}{\ee Zk}^T$. Then, the problem \eqref{EqTRSubProbl_s} above can be equivalently rewritten as
        \begin{equation}\label{EqTRSubProbl_p}
            \min_{u\in\R^m} \frac 12 u^T {\ee Zk}^T\ee Wk \ee Zk u+\skalp{{\ee Zk}^T\ee {z^*}k,u} \,, \quad\mbox{subject to}\quad\norm{\ee Zk u}\leq \ee\rho k \,.
        \end{equation}
    Any solution $s$ to \eqref{EqTRSubProbl_s} induces a solution $u=\big({\ee Zk}^T\ee Zk\big)^{-1}{\ee Zk}^Ts$ to \eqref{EqTRSubProbl_p} and, vice versa, $s={\ee Zk}u$ is a solution to \eqref{EqTRSubProbl_s} for every solution $u$ to \eqref{EqTRSubProbl_p}. Since \eqref{EqTRSubProbl_p} is a quadratic optimization problems, the use of a CG-method suggests itself. For its application consider a stopping tolerance $\ee{\bar \xi}k:=\chi(\norm{\ee {z^*}k})$, where $\chi:(0,\infty)\to(0,1)$ is a monotonically increasing function with $\lim_{t\downarrow 0}\chi(t)~=~0$. Then, we can apply the (preconditioned) CG-method for solving the system ${\ee Zk}^T\ee Wk \ee Zk u=-{\ee Zk}^T\ee {z^*}k$, generating iterates $u_j,p_j,\alpha_j$, $j=0,1,\ldots,$ with $u_0=0$ and $u_{j+1}=u_j+\alpha_jp_j$. We will stop the CG-method if
        \begin{equation*}
        \begin{split}
            \norm{{\ee Zk}^T\ee Wk \ee Zk u_j+{\ee Zk}^T\ee {z^*}k}\leq \ee{\bar \xi}k 
            &\mbox{ or }\norm{\ee Zk u_{j+1}}\geq\ee\rho k
            \\ &\mbox{ or } \, p_j^T {\ee Zk}^T\ee Wk \ee Zkp_j\leq 0\,.
        \end{split}
        \end{equation*}
    In the first case, we take $\ee sk=\ee Zk u_j$, while in the remaining two cases we can take $\ee sk=\ee Zk (u_j+\alpha p_j)$, where the step size $\alpha>0$ is chosen such that $\norm{\ee sk}=\ee \rho k$, as in the CG-Steihaug method \cite{Stei83}. For a different choice of $\ee sk$ in the latter two cases, we refer to \cite[Algorithm~2]{Gfr24}.
\end{enumerate}

Using this modified CG method, we obtain the following convergence result.

\begin{theorem}[{\cite[Corollary~5.5]{Gfr24}}]\label{ThConvCG} Consider the iterates generated by the BasGSSN method defined in Algorithm~\ref{algo_BasGSSN} with $\ee sk$ computed by the modified CG method as described above, and assume that the condition numbers of the matrices ${\ee Zk}^T\ee Zk$ are uniformly bounded. Furthermore, assume that $\varphi(\ee zk)$ is bounded from below and that the sequence $\ee xk$ has a limit point $\xb$ which is a local minimizer for $\varphi$. Moreover, assume that for any subspace $\rge(P,W)\in \Sp_{P,W}\partial\varphi(\xb,0)$ the matrix $W$ is positive definite and assume that $\partial\varphi$ is \SCD \ssstar at $(\xb,0)$. Then the iterates $\ee xk$ converge superlinearly to $\xb$.
\end{theorem}

Note that the uniform boundedness of the condition numbers of ${\ee Zk}^T\ee Zk$ can always be satisfied if, instead of arbitrary basis, an orthonormal basis for $\rge \ee Pk$ is chosen in the definition of the matrices $\ee Zk$, since then ${\ee Zk}^T\ee Zk = I$. Furthermore, note that compared with Theorem~\ref{ThSuperLinConv}, we now require a slightly stronger assumption, namely that for every subspace $(P,W)\in\Sp_{P,W}\partial\varphi(\xb,0)$ the matrix $W$ is not only nonsingular but also positive definite. This assumption is fulfilled for so-called \emph{tilt stable} local minimizers, see, e.g., \cite{PolRo98}. Moreover, in many cases the SC derivative
$\Sp\partial\varphi(\xb,0)$ will be a singleton, and in this case the nonsingularity of $W$ also implies its positive definiteness. The regularity condition that $\partial\varphi$ is \SCD \ssstar at $(\xb,0)$ is discussed in Section~\ref{sect_appl_sparsity_TV}.

\section{Application to sparsity and TV regularization}\label{sect_appl_sparsity_TV}

In this section, we explicitly describe the single tasks of the BasGSSN method, i.e., Algorithm~\ref{algo_BasGSSN}, for minimizing Tikhonov functionals $\Tad(x)$ with sparsity promoting ($\ell_p$) regularizers and total variation (TV) penalties, respectively. For ease of presentation, we restrict ourselves to the linear case $F(x)=Ax$, where $A$ is an $m\times n$ matrix, and
    \begin{equation}\label{prob_application}
        \min_{x\in\R^n} \, T_\alpha^\delta(x) := \norm{Ax - y^\delta}_{\R^m}^2 + \alpha R(x) \,,
    \end{equation}
where the penalty term $R$ is defined below depending on the setting (sparsity or TV).

\subsection{Application to sparsity regularization}\label{subsect_appl_sparsity}

Let $(\phi)_{i=1,\ldots,n}$ be an orthogonal basis of $\R^n$. The aim of sparsity regularization is to find an approximate solution of $Ax=y$ with as little nonzero coefficients $\spr{x,\phi_i}$ as possible. For this, the weighted $\ell_p$ norm, $p\in[0,1]$, is typically used as regularizer, i.e.,
    \begin{equation}\label{def_Rspf}
        \Rspf(x):=\sum_{i=1}^n w_i\vert \skalp{x,\phi_i}\vert^p \,,
        \qquad w_i > 0 \,.
    \end{equation}
A summary of the regularization properties of this approach can, e.g., be found in \cite{Scherzer_Grasmair_Grossauer_Haltmeier_Lenzen_2008}. For the minimization of the resulting Tikhonov functional $\Tad$, a number of different methods have been proposed. Among these we mention, e.g., surrogate functional approaches \cite{Ramlau_Teschke_2010,Ramlau_Teschke_2005,Ramlau_Teschke_2006}, nonlinear transformation approaches with Nemskii operators \cite{Ramlau_Zarzer_2012,Zarzer_2009, Hinterer_Hubmer_Ramlau_2020}, Krylov subspace techniques \cite{Buccini_Reichel_2019,Huang_Lanza_Morigi_Reichel_Sgallari_2017,Lanza_Morigi_Reichel_Sgallari_2015}, compressed sensing algorithms \cite{Candes_Romberg_Tao_2006,Daubechies_DeVore_Fornasier_Gunturk_2010,Donoho_Tanner_2005}, conditional gradient and semismooth Newton methods \cite{Bonesky_Bredies_Lorenz_Maass_2007,Bredies_Lorenz_Maass_2009,Griesse_Lorenz_2008}, R-regularized Newton schemes \cite{Hintermueller_Wu_2014}, as well as the well-known ISTA and FISTA algorithms \cite{Beck_Teboulle_2009,Daubechies_Defrise_DeMol_2004}.

For our purposes, let $\Phi$ denote the $n\times n$ matrix whose $i$-th column, $i=1,\ldots,n$,  equals to $\phi_i$. Using the variable transformation $x=\Phi v$, problem \eqref{prob_application} becomes
    \begin{equation}\label{EqTransfProbl}
        \min_{v\in\R^n} \norm{\tilde A v-y^\delta}_{\R^m}^2+\alpha\sum_{i=1}^nw_i\vert v_i\vert^p \,,
    \end{equation}
with $\tilde A=A\Phi$. This problem can be written in the form \eqref{EqOptProbl}, with 
    \begin{equation*}
        f(v)=\norm{\tilde A v - y^\delta}_{\R^m}^2 \,,
        \qquad \text{and} \qquad
        g(v)=\alpha\sum\limits_{i=1}^n w_i\vert v_i\vert^p \,.
    \end{equation*}
We now describe in detail the single tasks of  the BasGSSN method for solving \eqref{EqTransfProbl}:

\begin{enumerate}
    \item[1.] \textbf{Evaluation of $T_\lambda(\cdot)$:} Given $v=\ee vk$ and $\lambda=\ee\lambda k$, we have to compute $z=\ee zk\in T_\lambda(v)$. This means, that we have to compute the gradient $\nabla f(v)=2\tilde A^T(\tilde Av-y^\delta)$ and then to compute
        \begin{equation*}
            z\in\argmin_{z\in\R^n}\Kl{\frac 12\norm{z-\hat v}_{\R^n}^2+\alpha\lambda\sum_{i=1}^nw_i\vert z_i\vert^p} \,,
        \end{equation*}
    where $\hat v=v-\lambda\nabla f(v)$. This problem has a separable structure and can be solved componentwise, i.e., for each $i=1,\ldots,n$ we have to solve the one dimensional problem
        \begin{equation*}
            z_i\in \argmin_{z\in\R}\Kl{\frac 12(z-\hat v_i)^2+\alpha\lambda w_i\vert z\vert^p} \,.
        \end{equation*}
    This subproblem is easy to solve for $p\in\{0,1\}$ and for the values $p\in\{\frac 12,\frac 23\}$ there are still exact formulas available, cf. \cite{CaXuSu13}. For any other value of $p$ we have to apply a numerical procedure to approximate the global minimizer. Due to its special structure, this is still a tractable problem.

    Having determined $z$, we have at our disposal also one subgradient
        \begin{equation*}
            z^*_g =-\frac1\lambda(z-v)-\nabla f(v)\in\partial g(z) \,,
        \end{equation*}
    and one subgradient
        \begin{equation*}
            z^*=\nabla f(z)+z^*_g\in\partial T_\alpha^\delta(z) \,.
        \end{equation*}
    \item[2.] \textbf{Determination of the SCD semismooth* Newton direction $\ee sk$:} Consider the index set $\I=\{i\in{1,\ldots,n}\mv z_i\not=0\}$. According to Examples \ref{ExEll_1}-\ref{ExEll_0}, for computing an element in $\Sp_{P,W}\partial g(z,z^*_g)$ we can take diagonal matrices $P$, $W$ with
        \begin{equation*}
            P_{ii}=
            \begin{cases}
            1 \,, &\mbox{if $i\in \I$ \,,}
            \\
            0 \,, &\mbox{if $i\not\in \I$,}
            \end{cases}
            \qquad
            W_{ii}=
            \begin{cases} 
            0 \,,&\mbox{if $i\in \I$ and $p\in\{0,1\}$} \,, 
            \\
            \alpha w_i p(p-1)\vert z_i\vert^{p-2} \,, &\mbox{if $i\in \I$ and $p\in(0,1)$} \,,
            \\
            1 \,, &\mbox{if $i\not\in \I$} \,,
            \end{cases}
        \end{equation*}
    and obtain $\rge(P,W)\in\Sp_{P,W}\partial g(z,z^*_g)$. Hence, the unit vectors $e_i$, $i\in \I$ form an orthogonal  basis for $\rge P$. Let $\tilde A_\I$ denote the submatrix of $\tilde A$ whose columns are those of $\tilde A$ corresponding to the indices $i\in \I$, let $W_\I$ denote the diagonal matrix with entries $W_{ii}$, $i\in \I$, and let $z^*_\I$ denote the vector with components $z^*_i$, $i\in \I$. Then problem \eqref{EqTRSubProbl_p} now takes the form
        \begin{equation*}
            \min_{u\in\R^\I}\frac 12 u^T(2\tilde A_\I^T\tilde A_\I+W_\I)u+\skalp{z_\I^*,u} \,,\qquad\mbox{ subject to }\norm{u}\leq\rho \,.
        \end{equation*}
    For some (approximate) solution $\bar u$ of this problem, the search direction then is
        \begin{equation*}
            \ee sk=\sum_{i\in \I}\bar u_ie_i \,.
        \end{equation*}
\end{enumerate}

We now discuss SCD regularity for problem \eqref{EqTransfProbl}, which is required in Theorem~\ref{ThSuperLinConv}. For this, let $\bar v$ denote a solution of \eqref{EqTransfProbl}, set $\underline \I:=\{i\in\{1,\ldots,n\}\mv \bar v_i\not=0\}$, and
    \begin{equation*}
        \overline \I:=\underline \I\cup
        \begin{cases}
            \{i\in\{1,\ldots,n\}\mv \bar v_i=0, \vert\bar r_i\vert=\alpha w_i\} \,, &\mbox{if $p=1$} \,,
            \\
            \emptyset \,, &\mbox{if $p\in(0,1)$} \,,
            \\
            \{i\in\{1,\ldots,n\}\mv \bar v_i=\bar r_i=0\} \,, &\mbox{if $p=0$} \,,
        \end{cases}    
    \end{equation*}
where $\bar r:=\nabla f(\bar v)=2\tilde A^T(\tilde A\bar v-y^\delta)$. Then, by taking into account  Examples \ref{ExEll_1}-\ref{ExEll_0}, the subgradient mapping of the objective in \eqref{EqTransfProbl} is SCD-regular at $(\bar v,0)$ if and only if for every index set $\I$ with $\underline \I\subset \I\subset\overline \I$, and with diagonal matrices $P^\I$ and $W^\I$ given by
        \begin{equation*}
            P_{ii}^\I=
            \begin{cases}
            1 \,, &\mbox{if $i\in \I$ \,,}
            \\
            0 \,, &\mbox{if $i\not\in \I$,}
            \end{cases}
            \qquad
            W_{ii}^\I=
            \begin{cases} 
            0 \,,&\mbox{if $i\in \I$ and $p\in\{0,1\}$} \,, 
            \\
            \alpha w_i p(p-1)\vert \bar v_i\vert^{p-2} \,, &\mbox{if $i\in \I$ and $p\in(0,1)$} \,,
            \\
            1 \,, &\mbox{if $i\not\in \I$} \,,
            \end{cases}
        \end{equation*}
the matrix $2P^\I\tilde A^T\tilde AP^\I+W^\I$
is nonsingular. Note that for $p\in\{0,1\}$, this is equivalent to the condition that for $\underline \I\subset \I\subset \overline \I$, the restricted matrices $\tilde A_\I^T\tilde A_\I$ are non-singular.

\begin{remark}
Another common choice for the regularization functional $\Rc$ in \eqref{Tikhonov_infinite} are the norms $\norm{\cdot}_{\BspqRn}$ of some general Besov spaces $\BspqRn$. For $p=q < \infty$, these Besov spaces coincide with the Sobolev-Slobodeckij spaces $W^{s,p}(\Rn)$ as long as $s$ is not an integer, and if $p=q=2$ the Besov spaces $B_{2,2}^s(\Rn)$ coincide with the classical Sobolev spaces $H^s(\Rn)$; see e.g.~\cite{Adams_1970}. While the definition of the Besov norms is somewhat unwieldy, there exists an equivalent characterization using wavelets. For this, let $(\psi_\lambda)_{\lambda\in\Lambda}$ denote an orthonormal wavelet family in $\LtRn$ corresponding to a sufficiently regular multiresolution analysis of $\LtRn$. Then an equivalent norm to $\norm{\cdot}_{\BsppRn}$ is \cite{Adams_Fournier_1977,Meyer_1993,Gerth_Ramlau_2014}
    \begin{equation*}
        \norm{x}_{\BspRn} := \kl{ \sum_{\lambda \in \Lambda} 2^{p \abs{\lambda} \kl{s+n\kl{\frac{1}{2} - \frac{1}{p}}}}\abs{\spr{x,\psi_\lambda}_\LtRn}^p }^\frac{1}{p} \,,
    \end{equation*}
assuming that $s+n\kl{1/2 - 1/p} \geq 0$ and with $\abs{\lambda}$ denoting the scale of the wavelets. Truncating the infinite sum above and discretizing, we thus obtain a regularizer $R$ which 
directly corresponds to the sparsity penalty \eqref{def_Rspf} with suitably chosen weights $w_i$. 
\end{remark}

\subsection{Application to TV regularization}\label{subsect_appl_TV}

The aim of TV regularization is to find approximate solutions of a 2D inverse problem which are piecewise constant. In the infinite dimensional setting \eqref{Fx=y}, the $\BVl$ semi-norm with $l \in \N$ is often chosen as a regularizer for this task \cite{Scherzer_Grasmair_Grossauer_Haltmeier_Lenzen_2008}, i.e., for $\Omega \subset \R^2$,
    \begin{equation*}
		\Rc(x) := \Rl(x) := \abs{D^l x}(\Rt) 
        :=
        \sup\Kl{ \int_\Omega u \kl{ \nabla^l \cdot \boldsymbol{\phi} } \, \Big\vert \, \boldsymbol{\phi} \in C_0^\infty(\Omega;\R^{\mathcal{N}(l)}) \,, \abs{\boldsymbol{\phi}} \leq 1  } \,.
    \end{equation*}
For the minimization of the resulting Tikhonov functionals, a number of approaches have been proposed, including primal-dual algorithms, thresholding methods, Bregman iterations, splitting approaches, nonlocal TV, and quadratic programming based on smooth approximations of the regularizer. For an overview of methods see, e.g., \cite{Mueller_Siltanen_2012}.

In the finite dimensional setting \eqref{Gx=y}, i.e., $Ax=y$, TV regularization is typically realized as follows: Assume that the vector $x$ can be identified with a matrix representation $(x_{ij})$, $j=1,\ldots,n_1$, $i=1,\ldots,n_2$. Then the following regularizer is used:
    \begin{equation}\label{def_RTV}
        R(x) = \RTVf(x):=\sum_{i=1}^{n_2-1}\sum_{j=1}^{n_1-1}\vert x_{i,j+1}-x_{i,j}\vert+\vert x_{i+1,j}-x_{i,j}\vert \,.
    \end{equation}
Unfortunately, no efficient algorithm for the evaluation of the proximal mapping ${\rm prox}_{\lambda R}$ for this functional is known, and thus the BasGSSN method cannot be applied directly. A remedy for this situation is given by an Augmented Lagrangian method (ALM). 

Let $s:=2(n_1-1)(n_2-1)$ and let the $s\times n$ matrix $B$ be such that the vector $Bx$ has the components $x_{i,j+1}-x_{i,j}$ and $x_{i+1,j}-x_{i,j}$, $i=1,\ldots,n_2, j=1,\ldots,n_1-1$. Then the problem of minimizing \eqref{prob_application} is equivalent to the constrained problem
    \begin{equation}\label{EqALMProbl}
        \min_{x,v}\norm{Ax-y^\delta}_{\R^m}^2+\alpha\norm{v}_1 \,, \qquad \mbox{ subject to }Bx=v \,.
    \end{equation}
Given a penalty parameter $\sigma>0$, we can define the augmented Lagrangian
    \begin{equation*}
    \begin{split}
        &\Lag_\sigma:\R^n\times\R^s \times\R^s\to\R \,,
        \\
        &\quad (x,y,v^*) \mapsto \Lag_\sigma(x,y,v^*):=\norm{Ax-y^\delta}_{\R^m}^2+\alpha\norm{v}_1+\skalp{v^*,Bx-v}+\frac\sigma2\norm{Bx-v}_{\R^s}^2 \,.
    \end{split}
    \end{equation*}
With this, a typical ALM for solving \eqref{EqALMProbl} is given as follows:

\begin{algorithm}[\emph{Augmented Lagrangian}]\label{algo_augmented_lagrangian} Given $\ee {v^*}0\in\R^s$, repeat until convergence:
    \begin{itemize}[align = left]
        \item[\textbf{Step~1}:] Select $\ee \sigma k>0$.
        \item[\textbf{Step~2}:] Compute $(\ee xk,\ee vk)\in\argmin_{x,v}\Lag_{\ee \sigma k}(x,v,\ee{v^*}k)$.
        \item[\textbf{Step~3}:] Set $\ee {v^*}{k+1}:=\ee{v^*}k+\ee\sigma k(B\ee xk-\ee vk)$.
    \end{itemize}
\end{algorithm}

A typical update rule for the penalty parameter $\ee \sigma k$, $k>0$, is given by
    \begin{equation*}
        \ee \sigma k=
        \begin{cases}
        \tau\ee\sigma{k-1} \,, & \mbox{if \ $\norm{B\ee xk-\ee vk}_{\R^s}>\xi\norm{B\ee x{k-1}-\ee v{k-1}}_{\R^s}$} \,,
        \\
        \ee\sigma{k-1} \,, &\mbox{else}\,,
        \end{cases}
    \end{equation*}
where $0<\xi<1<\tau$. Tracing back to \cite{He69,Po69}, the ALM is capable of tackling large-scale constrained problems. For some recent developments on this topic, see, e.g,~\cite{Be96, BiMa14}.

For our problem, we now apply the BasGSSN method to solve the auxiliary problem 
    \begin{equation*}
        \min_{x,v}\Lag_{\ee \sigma k}(x,v,\ee{v^*}k) \,.
    \end{equation*}
Given $v^*\in\R^s$ and $\sigma>0$, the augmented Lagrangian can be written as 
    \begin{equation}\label{EqAugmented}
        \Lag_\sigma(x,v,v^*)=f(x,v)+g(x,v) \,,
    \end{equation}
where
    \begin{equation*}
    \begin{split}
        f(x,v) &:=\norm{Ax-y^\delta}_{\R^m}^2+\skalp{v^*,Bx-v}+\frac\sigma2\norm{Bx-v}_{\R^s}^2\,,
        \\
        g(x,v)&:=\alpha\norm{v}_1=\alpha\sum_{i=1}^s\vert v_i\vert \,.
    \end{split}
    \end{equation*}
The detailed single tasks of the BasGSSN method for solving \eqref{EqAugmented} are as follows:

\begin{enumerate}
    \item[1.] \textbf{Evaluation of $T_\lambda(\cdot)$:} Given $(x,v)=(\ee xk,\ee vk)$ and $\lambda=\ee\lambda k$, we have to compute $z=(z_x,z_v)=\ee zk\in T_\lambda(x,v)$. This means that we have to compute the gradients 
        \begin{equation*}
        \begin{split}
            \nabla_x f(v)&=2 A^T(Ax-y^\delta)+B^T(v^*+\sigma(Bx-v)) \,,
            \\
            \nabla_v f(x,v)&=-(v^*+\sigma(Bx-v)) \,,
        \end{split}
        \end{equation*} 
    and then to compute $z_x=x-\lambda\nabla_xf(x,v)$, as well as
        \begin{equation*}
            z_v\in\argmin_{z\in\R^n}\Kl{\frac 12\norm{z-\hat v}_{\R^s}^2+\alpha\lambda\sum_{i=1}^n\vert z_i\vert} \,,
        \end{equation*}
    where $\hat v=v-\lambda\nabla_v f(x,v)$. Thus,
        \begin{equation*}
            z_v=\big(\max\{\vert\hat v_i\vert-\alpha\lambda,0\}{\rm sign\,}(\hat v_i)\big)_{i=1}^s \,.
        \end{equation*}
    Having determined $z$, we have at our disposal also one subgradient,
        \begin{equation*}
        \begin{split}
            &z^*=(z^*_x,z^*_v)
            \\ &= \kl{ \nabla_x f(z)-\frac{(z_x-x)}{\lambda}-\nabla_x f(x,v), \nabla_v f(z)-\frac{(z_v-v)}{\lambda} -\nabla_v f(x,v) } \in\partial T_\alpha^\delta(z) \,.
        \end{split}
        \end{equation*}
    \item[2.] \textbf{Determination of the SCD semismooth* Newton direction $\ee sk$:} In order to determine $\ee sk=(s_x,s_v)$, let $\I=\{i\in{1,\ldots,s}\mv (z_v)_i\not=0\}$. According to Example \ref{ExEll_1}, we can take $P$ and $W$ as diagonal matrices with entries
        \begin{equation*}
            P_{ii}=
            \begin{cases}
            1\,,&\mbox{if $i\in \I$}\,,
            \\
            0\,,&\mbox{if $i\not\in \I$\,,}
            \end{cases}
            \qquad
            W_{ii}=
            \begin{cases}
            0\,,&\mbox{if $i\in \I$}\,,
            \\
            1 \,, &\mbox{if $i\not\in \I$} \,.
            \end{cases}
        \end{equation*}
    Then, the Newton direction can be computed as an (approximate) solution of 
        \begin{equation*}
        \begin{split}
            \min_{s_x,s_v} \, \frac 12\begin{pmatrix}s_x^T& s_v^T\end{pmatrix}
            \begin{pmatrix}2A^TA+\sigma B^TB&-\sigma B^T\\-\sigma B&\sigma I+W\end{pmatrix}
            \begin{pmatrix}s_x\\s_v\end{pmatrix}
            +\skalp{z_x^*,s_x}+\skalp{z^*_v,s_v} \,,
        \end{split}
        \end{equation*}
    subject to the constraints $(s_v)_i=0$ for $i\not\in \I$, and $\norm{(s_x,s_v)}_{\R^n\times\R^s}\leq \rho$. 
\end{enumerate}

As before, we now discuss SCD regularity for problem \eqref{EqAugmented}. For this, let $\bar x, \bar v$ denote a solution of the auxiliary problem \eqref{EqAugmented}, set $\underline \I:=\{i\in\{1,\ldots,s\}\mv \bar v_i\not=0\}$, and
    \begin{equation*}
        \overline \I:=\underline \I\cup\{i\in\{1,\ldots,s\}\mv \vert \bar r_i\vert=\alpha\} \,,
        \qquad \text{where} \qquad 
        \bar r=\nabla_v f(\bar x,\bar v) \,.
    \end{equation*}
Next, for index sets $\underline \I\subset \I\subset\overline \I$, let the $s\times s$ diagonal matrices $P^I,\ W^I$ be defined by
    \begin{equation*}
        P^\I_{ii}=
        \begin{cases}
            1 \,, &\mbox{if $i\in \I$} \,,
            \\
            0 \,, & \mbox{if $i\not\in \I$} \,, 
        \end{cases}
        \qquad \text{and} \qquad
        W^\I=
        \begin{cases}
            0 \,, & \mbox{if $i\in \I$} \,,
            \\
            1\,, &\mbox{if $i\not\in \I$} \,.
        \end{cases}
    \end{equation*}
Then the auxiliary problem is SCD-regular at $\bar x,\bar v$, if and only if the matrix
    \begin{equation*}
        \begin{pmatrix}
            A^TA+\sigma B^TB&-\sigma B^T P^\I\\
            -\sigma P^\I B&\sigma P^\I+W^\I
            \end{pmatrix}
    \end{equation*}
is nonsingular for every index set $\I$ with $\underline \I\subset \I\subset\overline \I$.

\section{Numerical Examples}\label{sect_numerics}

\begin{figure}[ht!]
    \centering
    \includegraphics[width=\textwidth, trim = {8cm 4cm 6.5cm 1cm}, clip=true]{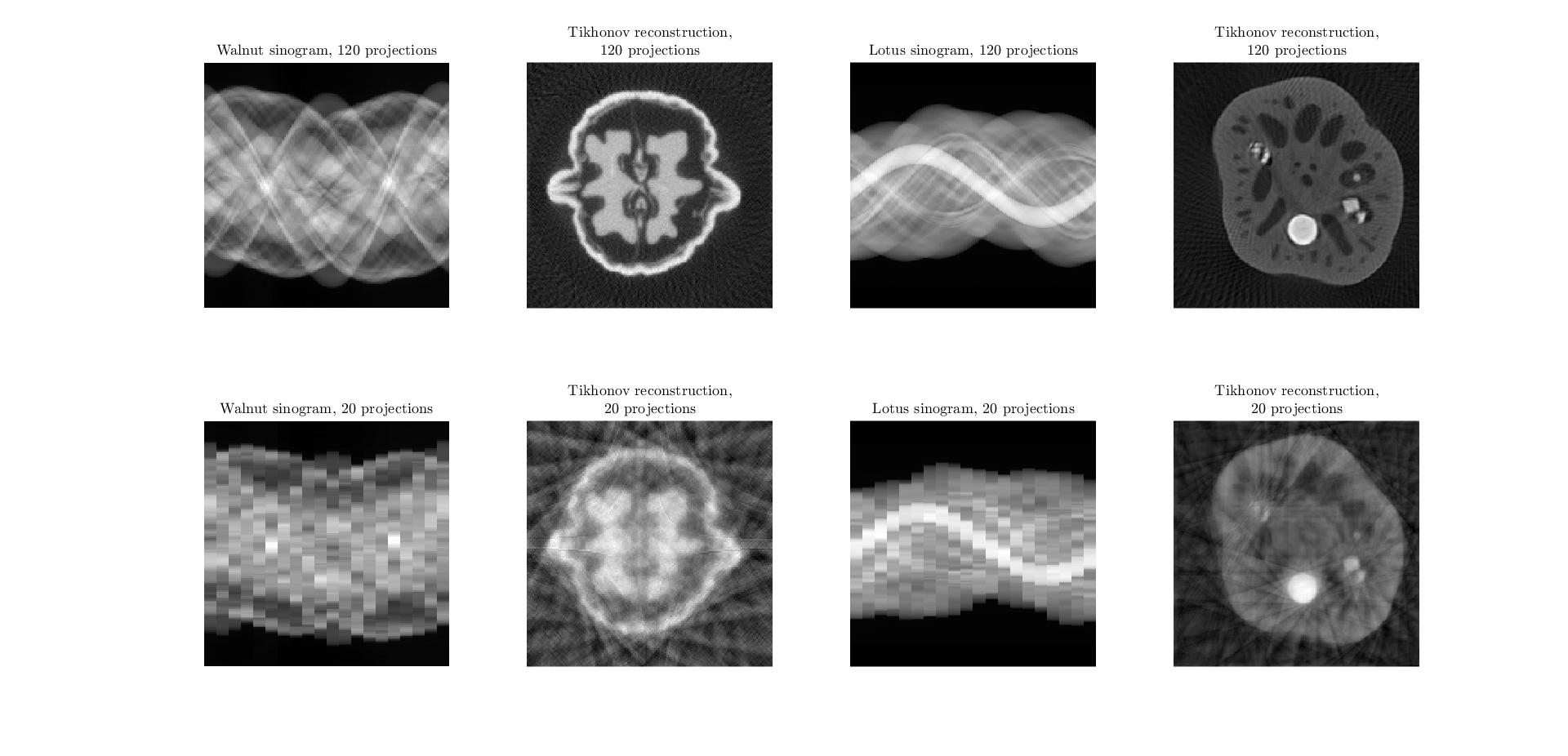}
    \caption{Test data: sinograms and reference reconstructions adapted from \cite{FIP_Lotus_2016,FIP_Walnut_2015}.}
    \label{fig_ground_truth}
\end{figure}

In this section, we illustrate the performance of the SCD semismooth* Newton method on examples from tomographic imaging based on the (2D) Radon transform \cite{Natterer_2001,Louis_1989}
    \begin{equation}\label{Radon_A}
    \begin{split}
        (\mathcal{A} u)(\sigma,\theta) :=  
        \int_\R u(\sigma \omega(\theta) + \tau \omega(\theta)^\perp) \, d\tau \,,
    \end{split}	
    \end{equation}
where $\omega(\theta) = (\cos(\theta),\sin(\theta))^T$ for $\theta \in [0,2\pi)$ and $\sigma \in \R$. For our examples, we use the walnut and lotus datasets recorded by the Finnish Inverse Problems Community \cite{FIP_Lotus_2016,FIP_Walnut_2015}, which contain sinograms of different resolution, as well as corresponding matrix representations $A$ of the Radon transform $\mathcal{A}$. These sinograms, as well as reference reconstructions using standard $\ell_2$ Tikhonov regularization, are depicted in Figure~\ref{fig_ground_truth}.

In our tests, we consider a resolution of $160 \times 160$ and $256 \times 256$ pixels for the unknown density in the walnut and lotus case, respectively. Furthermore, as in \cite{FIP_Lotus_2016,FIP_Walnut_2015}, we consider both $120$ and $20$ uniformly distributed angles (projections), and refer to these cases as the full and limited-angle problems, respectively. For reconstruction, we then consider the Tikhonov functional \eqref{prob_application}, where for the regularizer $R$ we use both the sparsity ($\ell_p$) and TV penalties defined in \eqref{def_Rspf} and \eqref{def_RTV}, respectively. In the sparsity penalty, we use constant weights $w_i=1$, orthonormal Daubechies wavelets \cite{Daubechies_1992} of maximal order for the functions $\phi_i$, as well as the choices $p=1$, $p=0.5$, and $p=0$. Since the aim of these tests is to illustrate the performance of our SCD semismooth* Newton method, and not to establish one regularizer as superior to another, the regularization parameter $\alpha$ in \eqref{prob_application} was chosen manually among a range of tested values. All computations were made in Matlab on a standard desktop computer.

\begin{figure}[ht!]
    \centering
    \includegraphics[width=\textwidth, trim = {8cm 4cm 6.5cm 1cm}, clip=true]{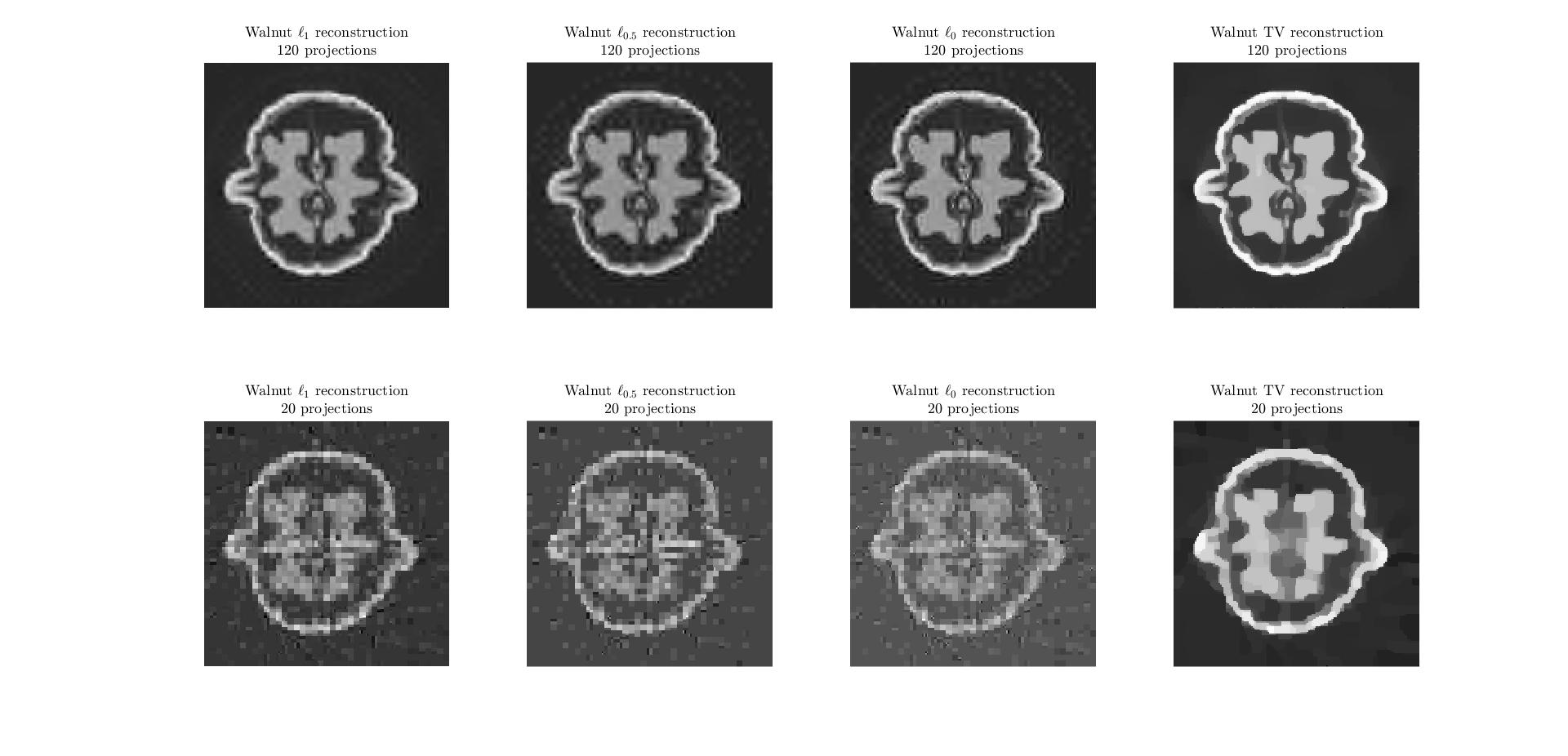}
    \caption{SCD semismooth* Newton reconstructions for the walnut data.}
    \label{fig_walnut_results}
\end{figure}

\begin{figure}[ht!]
    \centering
    \includegraphics[width=\textwidth, trim = {8cm 4cm 6.5cm 1cm}, clip=true]{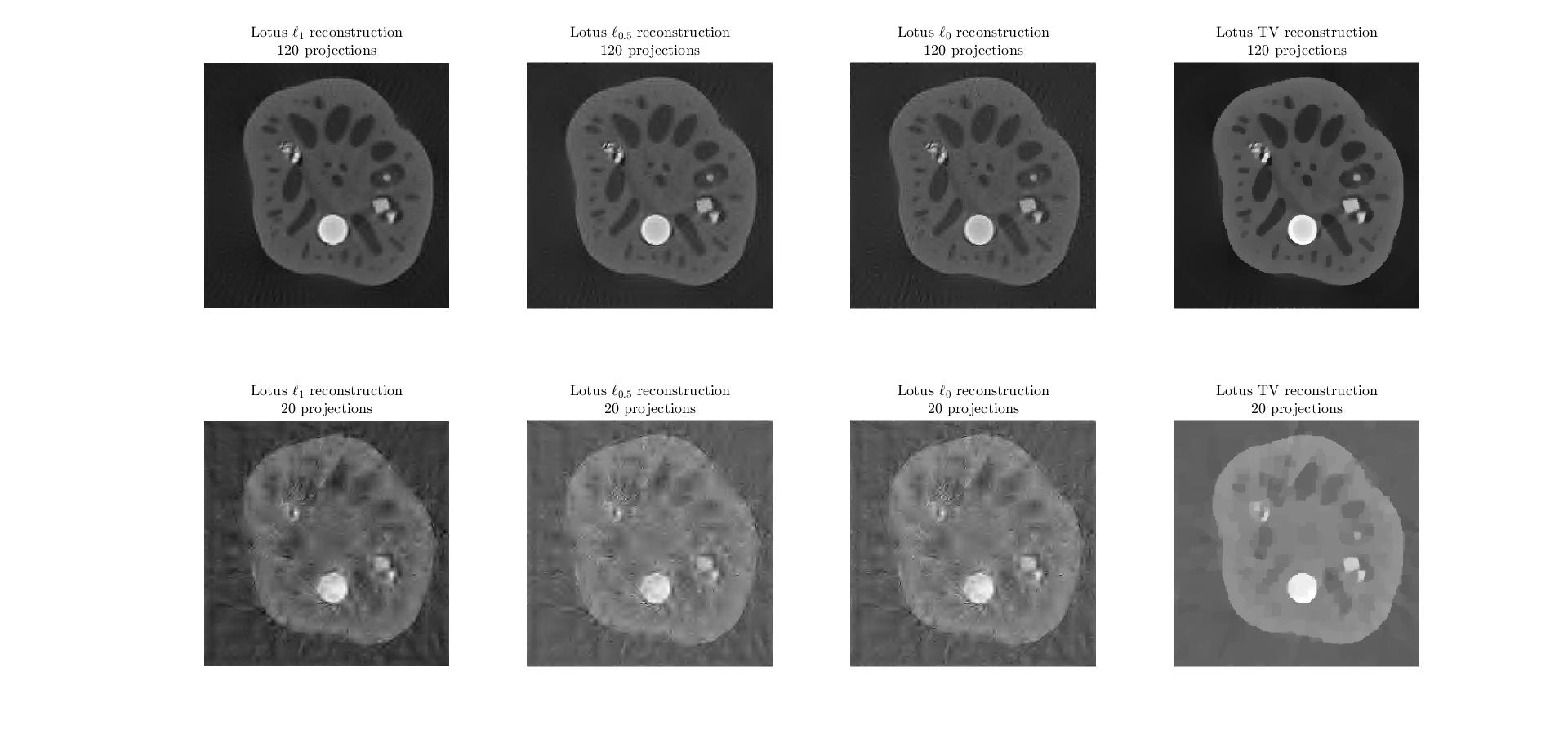}
    \caption{SCD semismooth* Newton reconstructions for the lotus data.}
    \label{fig_lotus_results}
\end{figure}

\begin{figure}[ht!]
    \centering
    \includegraphics[width=\textwidth, trim = {7cm 0.5cm 6cm 0.5cm}, clip=true]{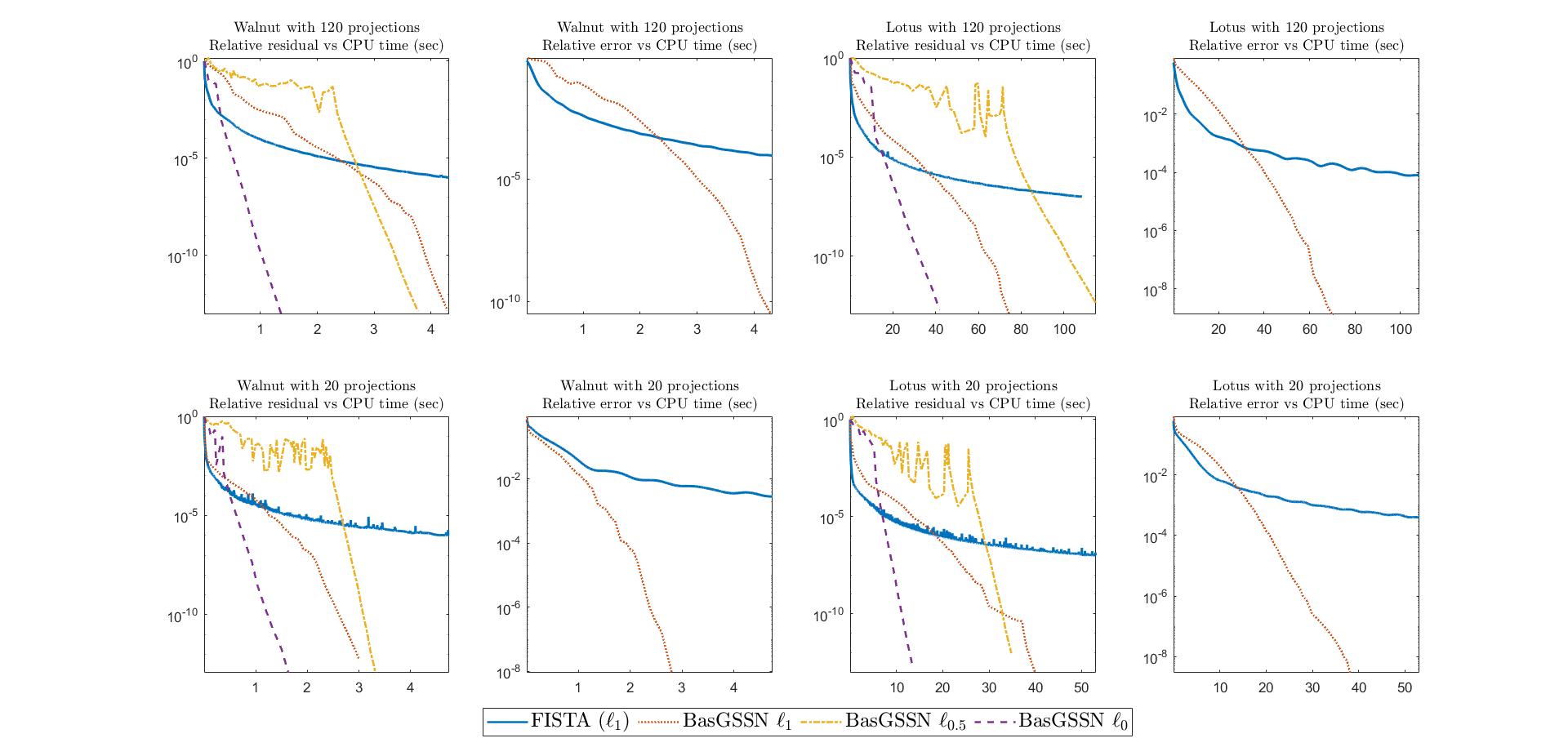}
    \caption{Illustration of the computational efficiency of the SCD semismooth* Newton method for the case of sparsity regularization, compared with classic FISTA (for $\ell_1$) \cite{Beck_Teboulle_2009}.}
    \label{fig_accuracy}
\end{figure}

Figure~\ref{fig_walnut_results} and Figure~\ref{fig_lotus_results} depict the results for the walnut and lotus tests, respectively, obtained with the SCD semismooth* Newton method as described in Section~\ref{sect_semismooth_Tikhonov}. In the full angle case, all regularization methods yield suitable results, with only minor structural differences being visible when comparing the results obtained with the different chosen penalties. As expected, TV regularization leads to piecewise constant reconstructions, which is clearly visible in the limited-angle case, where due to the higher ill-posedness of this setting a larger regularization parameter was chosen compared to the full angle case. This increased ill-posedness is also reflected in the quality of the obtained reconstructions using the sparsity penalty, where the different choices of $p$ perform similarly, but with minor notable differences. Nevertheless, in all cases the SCD semismooth* Newton method is able to accurately minimize the corresponding Tikhonov functional \eqref{prob_application}. This can be seen in Figure~\ref{fig_accuracy}, which illustrates the accuracy of our method in terms of both the relative residual and the relative error (for $p=1$) in the case of sparsity regularization. In particular, note that within the same or shorter computation (CPU) time, the SCD semismooth* Newton method is several orders of magnitude more accurate than the classic FISTA algorithm (for $p=1$) \cite{Beck_Teboulle_2009}.

\section{Conclusion}\label{sect_conclusion}

In this paper, we considered the SCD semismooth* Newton method for the efficient numerical minimization of Tikhonov functionals with non-smooth and non-convex penalty terms. Furthermore, we presented a globalized version of this locally superlinearly convergent method, as well as explicit algorithms for the case of Tikhonov regularization with sparsity and TV promoting regularizers. Finally, we numerically illustrated the performance of these methods on (limited-angle) tomographic imaging problems.

\section{Support}

This research was funded in part by the Austrian Science Fund (FWF) SFB 10.55776/F68 ``Tomography Across the Scales'', project F6805-N36 (Tomography in Astronomy). For open access purposes, the authors have applied a CC BY public copyright license to any author-accepted manuscript version arising from this submission.

\appendix

\section{The SCD derivative of $\ell_p$ norms with $0\leq p \leq 1$}\label{appendix_examples}

\begin{example}\label{ExEll_1}
Let $\varphi(x) := \norm{x}_1=\sum_{i=1}^n\vert x_i\vert$ be the $\ell_1$ norm on $\R^n$. Then,
    \begin{equation*}
        \partial \varphi(x) =  \partial \norm{\cdot}_1(x)=\partial \vert \cdot\vert(x_1)\times\partial \vert \cdot\vert(x_2)\times\ldots\times\partial \vert \cdot\vert(x_n) \,,
    \end{equation*}
and due to \cite[Lemma 3.8]{GfrOutVal22b} there holds
    \begin{align*}
        &\lefteqn{\Sp\partial \norm{\cdot}_1(x,x^*)=}\\& \Kl{\big\{(u,u^*)\in\R^n\times\R^n\mv (u_i,u_i^*)\in L_i,\ i=1,\ldots,n\big\}\Bmv L_i\in\Sp \partial\vert\cdot\vert(x_i,x_i^*),\ i=1,\ldots,n } \,.
    \end{align*}
Now, since
    \begin{equation*}
        \gph \partial \vert\cdot\vert=\big(\R_-\times \{-1\}\big)\cup\big(\{0\}\times[-1,1]\big)\cup\big(\R_+\times\{+1\}\big) \,,
    \end{equation*}
we obtain
    \begin{equation*}
        \OO_{\partial\vert\cdot\vert}=\big(\R_{--}\times \{-1\}\big)\cup\big(\{0\}\times(-1,1)\big)\cup\big(\R_{++}\times\{1\}\big) \,,
  \end{equation*}
as well as
    \begin{equation*}
        \Sp \partial\vert\cdot\vert(t,t^*)=
        \begin{cases}
            \{\R\times \{0\}\} \,, &\mbox{if $(t,t^*)\in \big(\R_{--}\times \{-1\}\big)\cup \big(\R_{++}\times\{+1\}\big)$\,,}
            \\
            \{\{0\}\times\R\} \,, &\mbox{if $(t,t^*)\in\{0\}\times(-1,1)$\,.}\end{cases}   
    \end{equation*}
For the remaining two points $(0, \pm 1)$, we have to take limits of subspaces, which yields
    \begin{equation*}
        \Sp \partial\vert\cdot\vert(0, \pm 1)=\{\R\times \{0\}, \{0\}\times\R\} \,.
    \end{equation*}
Hence, it follows that every $L\in \Sp \partial\norm{\cdot}_1(x,x^*)$ has the representation $L=\rge(P,I-P)$, where $P$ is a diagonal matrix with diagonal entries belonging to $\{0,1\}$.
\end{example}

\begin{example}\label{ExEll_q}
Let $\varphi(x) := \norm{x}_q^q =\sum_{i=1}^n\vert x_i\vert^q$ for $q\in(0,1)$. Then,
    \begin{equation*}
        \partial \varphi(x) = \partial \norm{\cdot}_q^q(x)=\partial \vert \cdot\vert^q(x_1)\times\partial \vert \cdot\vert^q(x_2)\times\ldots\times\partial \vert \cdot\vert^q(x_n) \,,
    \end{equation*}
where 
    \begin{equation*}
        \partial \vert \cdot\vert^q(t)=
        \begin{cases}
            \{q\vert t\vert ^{q-1}{\rm sign\,}(t)\} \,, &\mbox{if $t\not=0$} \,,
            \\ \R \,, &\mbox{if $t=0$} \,.
        \end{cases}     
    \end{equation*}
Hence, we find that
    \begin{equation*}
        \Sp \partial\vert\cdot\vert^q(t,t^*)=\{\rge(p(t),w(t))\}\mbox{ with }(p(t),w(t)):=
        \begin{cases}
            (1,q(q-1)\vert t\vert^{q-2}) \,, &\mbox{if $t\not=0$} \,, 
            \\
            (0,1) \,, &\mbox{if $t=0$} \,.
        \end{cases}    
    \end{equation*}
Again using \cite[Lemma~3.8]{GfrOutVal22b}, we obtain that
    \begin{align*}
        \lefteqn{\Sp\partial \norm{\cdot}_q^q(x,x^*)=}
        \\
        &\Kl{\big\{(u,u^*)\in\R^n\times\R^n\mv (u_i,u_i^*)\in L_i,\ i=1,\ldots,n\big\}\Bmv L_i\in\Sp \partial\vert\cdot\vert^q(x_i,x_i^*),\ i=1,\ldots,n }
    \end{align*}
and thus $\Sp\partial \norm{\cdot}_q^q(x,x^*)$ consists of the single subspace $\rge(P,W)$, where $P$ and $W$ are diagonal matrices satisfying
    \begin{equation*}
        P_{ii}=
        \begin{cases}
            1 \,, &\mbox{if $x_i\not=0$}\,,
            \\
            0 \,, &\mbox{if $x_i=0$} \,,
        \end{cases}
        \qquad \text{and} \qquad
        W_{ii}=
        \begin{cases}
            q(q-1)\vert x_i\vert^{q-2} \,, &\mbox{if $x_i\not=0$} \,,
            \\
            1 \,, & \mbox{if $x_i=0$} \,.
        \end{cases}    
    \end{equation*}
    
Furthermore, note that $\vert\cdot\vert^q$ is prox-regular (and subdifferentially continuous) at every $\bar t\in\R$ for every $\bar t^*\in\partial\vert \cdot\vert^q(\bar t)$. This is clear when $\bar t\not=0$. In the case when $\bar t=0$ and $\bar t^*\in\partial\vert \cdot\vert^q(0)=\R$, we can find $\epsilon>0$ sufficiently small such that 
    \begin{equation*}
        (B_\epsilon(0)\times\B_\epsilon(\bar t^*))\cap \gph \partial\vert\cdot\vert^q=\{0\}\times\B_\epsilon(\bar t^*) \,.
    \end{equation*}
Hence, the only $t,t'$ which we have to consider according to the definition of prox-regularity, are $t=t'=0$ and thus condition \eqref{EqProxRegDef} holds with $\rho=0$. Furthermore, $\vert \cdot\vert^q$ is subdifferentially continuous at $0$ for arbitrary $\bar t^*\in\partial\vert \cdot\vert^q(0)$, since any sequence $(t_k,t_k^*)$ converging in $\gph\partial\vert\cdot\vert^q$ to $(0,\bar t^*)$ satisfies $t_k=0$ for all $k$ sufficiently large.

Note also that the functions $p(\cdot),w(\cdot)$ are not continuous at $0$, and that $w(\cdot)$ is even unbounded. Nevertheless, there holds 
    \begin{equation*}
        \lim_{t\to0} \, \rge(p(t),w(t))=\rge(p(0),w(0)) 
        \qquad \text{in} \, \Z_1\,,
    \end{equation*}
and all subspaces $\rge(p(t),w(t))$ for $t\in\R$ belong to $\Z_1^{\rm reg}$. From this, we can see that the $(P,W)$-basis representation of subspaces can lead to unbounded matrices $W$. However, for the convergence theory of the SCD semismooth* Newton method this is of no relevance. It is only a numerical issue for computing the Newton direction, which we circumvent via scaling in order to avoid ill-conditioned systems of linear equations.
\end{example}

\begin{example}\label{ExEll_0}
Let $\varphi(x) = \norm{x}_0:=\sum_{i=1}^n \xi(x_i)$ be the $\ell_0$ pseudo-norm on $\R^n$, where
    \begin{equation*}
        \xi(t):=\begin{cases}1 \,, &\mbox{if $t\not=0$}\,,\\0 \,, &\mbox{if $t=0$} \,. \end{cases} 
    \end{equation*}
Due to \cite[Proposition 10.5]{RoWe98}, there holds
    \begin{equation*}
        \partial \norm{x}_0=\partial\xi(x_1)\times\partial\xi(x_2)\times\ldots\times\partial\xi(x_n) \,.    
    \end{equation*}
Since $\gph \partial\xi =\big(\R\times\{0\}\big)\cup \big(\{0\}\times\R\big)$, we thus obtain
    \begin{equation*}
        \OO_{\partial\xi}=\big((\R\setminus\{0\})\times\{0\}\big)\cup \big(\{0\}\times(\R\setminus\{0\})\big) \,,    
    \end{equation*}
as well as
    \begin{equation*}
        \Sp\partial\xi(t,t^*)=
        \begin{cases}
            \big\{\R\times\{0\}\big\} \,, &\mbox{if $t\in\R\setminus\{0\}$, $t^*=0$\,,}
            \\
            \big\{\{0\}\times\R\big\} \,, &\mbox{if $t=0$, $t^*\in\R\setminus\{0\}$ \,. }
        \end{cases} 
    \end{equation*}
The SC derivative at $(0,0)$ is again found by a limiting process, which results in
    \begin{equation*}
        \Sp\partial\xi(0,0)=\big\{\R\times\{0\},\{0\}\times\R\} \,.    
    \end{equation*}
With this, we can now conclude together with \cite[Proposition 3.5]{GfrManOutVal22} that
    \begin{align*}
        \lefteqn{\Sp\partial \norm{\cdot}_0(x,x^*)=}
        \\
        &\Kl{\big\{(u,u^*)\in\R^n\times\R^n\mv (u_i,u_i^*)\in L_i,\ i=1,\ldots,n\big\}\Bmv L_i\in\Sp \partial\xi(x_i,x_i^*),\ i=1,\ldots,n}\,.
    \end{align*}
Note that $\xi$ is not subdifferentially continuous at $0$ for $0$ and therefore, $\norm{\cdot}_{0}$ is not subdifferentially continuous at $x$ for $x^*$ whenever $x_i=x_i^*=0$ for some $i$. Nevertheless, for all subspaces $L\in \Sp\partial \norm{\cdot\nobreak}_0(x,x^*)$ we have $L=L^*=\rge(P,I-P)\in\Z_n^{P,W}$, where $P$ is a diagonal matrix with diagonal entries belonging to $\{0,1\}$, i.e., the subspaces have exactly the same form as in case of the $\ell_1$-norm.
\end{example}

\bibliographystyle{plain}
{\footnotesize
\bibliography{mybib}
}

\end{document}